\documentclass{amsart}
\usepackage{amsmath,amssymb}
\usepackage[toc,page]{appendix}
\usepackage{amsmath}
\usepackage{enumitem}

\usepackage{color}
\numberwithin{equation}{section}

\theoremstyle{plain}

  \newtheorem{thm}{Theorem}[section]

\theoremstyle{remark}
\newtheorem*{proof1}{Proof of Theorem \ref{thm}}
\newtheorem*{proof2}{Proof of Theorem \ref{thm}}
\newtheorem*{prooflbwb}{Proof}

\newtheorem*{proofwb}{Proof}
\newtheorem*{prooflbcwb}{Proof}
\newtheorem*{prooflhwb}{Proof}
\newtheorem*{proofsmpwb}{Proof}
\theoremstyle{defn}
\newtheorem{prop}[thm]{Proposition}
\newtheorem{lem}[thm]{Lemma}
\newtheorem{cor}[thm]{Corollary}

\newtheorem{oss}{Remark}
\newtheorem*{step1}{Step $1$}
\newtheorem*{step2}{Step $2$}
\newtheorem*{step3}{Step $3$}
\newtheorem*{step4}{Step $4$}
\newtheorem*{step5}{Step $5$}

\newtheorem{defn}{Definition}[section]

\def\d{\delta}
\def\to{\rightarrow}
\def\eps{\varepsilon}

\newcommand{\R}{{\mathbb R}}

\title[On Neumann type problems for nonlocal HJ equations]{On Neumann problems for nonlocal Hamilton-Jacobi equations with dominating gradient terms}
\thanks{
	This work was partially supported by the ERC advanced grant $668998$ (OCLOC) under the EU's H$2020$ research programme. }

\author{Daria Ghilli   }
 \address{Daria Ghilli, University of Graz,Institute of Mathematics and Scientific Computing,  Universit{\"a}tsplatz 3, Austria}
\email{daria.ghilli@uni-graz.at}

\begin{document}
\maketitle

 
\begin{abstract}
We are concerned with the well-posedness of  Neumann boundary value
problems for nonlocal Hamilton-Jacobi  equations related to jump processes in general smooth domains. We consider a nonlocal diffusive term of censored type of order strictly less than $1$ and Hamiltonians both in coercive form and in noncoercive  Bellman form, whose growth in the gradient make them the leading term in the  equation. We prove a comparison principle for bounded sub-and supersolutions in the context of viscosity solutions with generalized boundary conditions, and consequently by Perron's method we get the existence and uniqueness
of continuous solutions.
We give some applications in the evolutive setting, proving  the large time behaviour of the associated evolutive problem under suitable
assumptions on the data.
\end{abstract}

\section{Introduction}\label{sec:bounded}
The aim of this work is the analysis of the well-posedness of  Neumann boundary value problems for partial-integro differential equations (PIDEs in short)  of Hamilton-Jacobi type, where the nonlocal terms are singular integrals related to the infinitesimal generator of discontinuous jump processes.  To be more specific, we consider the following 
\begin{equation}\label{solutions0}
\left\{
 \begin{array}{ll}
 u(x)-\mathcal{I}[u](x)+H(x,Du)=0 \, \, &\mbox{in} \, \, \Omega, \\
                 \frac{\partial u}{\partial n}= 0 \, \, &\mbox{on} \, \, \partial \Omega,
                 \end{array}
\right.\,
\end{equation}
where  $H\, : \, \bar{\Omega} \times \mathbb{R}^N \mapsto \mathbb{R}$ is a continuous function, $\Omega\subset \mathbb{R}^N$ is an open (smooth enough) domain,  $n(x)$ is the exterior unit normal vector to $\partial \Omega$ at $x \in \partial \Omega$
and $\mathcal{I}[u]$ is an \textit{integro-differential operator of censored type and of order strictly less than $1$} (see \eqref{idol} for the definition).

In the probabilistic approach to PDEs, Neumann boundary conditions are associated to stochastic processes being reflected on the boundary. The underlying idea  is to force the stochastic process to remain inside the domain of the equation. 
Classically, this is obtained essentially by a reflection on the boundary (see  the method developed by Lions and  Sznitman \cite{LS} in the continuous setting).
A key result  in the classical  setting is that, for a
PDE with Neumann boundary conditions, there is a unique underlying
reflection process  and any consistent approximation will converge to it
(see \cite{LS} and Barles, Lions \cite{BL}). 

When dealing with discontinuous jump processes, the underlying idea is the same but the situation is different. This is essentially due to the fact that the jump process  may exit the domain without having first hit the boundary. The consequence is that Neumann boundary conditions can be obtained in many ways, depending on the kind of reflection we impose on the outside jumps. 
Moreover, the choice of 
a reflection on the boundary changes the equation inside the domain.

The starting point of our work is the paper \cite{BCGJ} where Barles, Chasseigne, Georgeline and Jakobsen  studied problems as in \eqref{solutions0} in the case of linear equations (that is, without the Hamiltonian $H$) and when the domain $\Omega$ is the half space. In  \cite{BCGJ} different models of reflection are presented, among which two types of reflections are particularly relevant for possible extensions in a more general setting. The first is the  \textit{normal projection}, close to the approach of Lions-Sznitman in \cite{LS}, where  outside jumps are immediately projected to the boundary
by killing their normal component. This  model has been thoroughly investigated
in the paper \cite{BGJ} for fully non-linear equations set in general domains. 

The second, the  \textit{censored} model, is the one we  consider in our paper. In this case, any outside jump of the underlying process is
cancelled (censored) and the process is restarted (resurrected) at the origin of that
jump. 

In particular, in the present work, we consider the boundary value problem \eqref{solutions0} where  $\mathcal{I}[u]$ is an \textit{integro-differential operator of censored type and of order strictly less  than  $1$} defined as  
\begin{equation}\label{idol}
\mathcal{I}[u](x)= \lim_{\d \to 0^+}\int_{\begin{matrix} \mbox{\footnotesize{$|z|>\d$}}, \\ \mbox{\footnotesize{$x+j(x,z)\in \bar{\Omega}$}}\end{matrix}} \, [u(x+j(x,z))-u(x)]d\mu_x(z),
\end{equation}
where $\mu_x$ is a singular  non-negative Radon measure representing the intensity of the jumps from $x$ to $x+z$  and satisfying the following integrability condition 
$$
\int |z|\wedge 1 d\mu_x(z) < + \infty,
$$
and $j(x,z)$ is a jump function (see assumptions (M), (J0), (J1) in the following section for details). A meaningful example is
\begin{equation}\label{muintro}
d\mu_x(z)\sim \frac{g(x,z)}{|z|^{N+\sigma}}dz, \quad \sigma \in (0,1),\quad j(x,z)=f(x) z,
\end{equation} 
where $g, f$ are bounded and Lipschitz functions. 
Note that $\mathcal{I}$ has to be interpreted as a principal value (P.V.) integral. We remark that the domain of integration is restricted to the $z$ such that $x+j(x,z) \in \bar{\Omega}$, avoiding thus any outside jump. Note also that, as a consequence of the fact that censored type processes are not allowed to jump outside $\bar{\Omega}$, we don't need any conditions on $\Omega^c$ in the boundary value problem  \eqref{solutions0}. 

We follow the  PIDE analytical approach developed in \cite{BCGJ},  in the sense that we  directly work with the infinitesimal generator and not yet with the  processes themselves.  For more details and probabilistic references on censored  processes, we refer to e.g. \cite{BBC}, \cite{FR}, \cite{J}, \cite{GM} and to the introduction of \cite{BCGJ}. We just mention that the underlying processes in this paper are related to the censored stable processes of Bogdan \cite{BBC} and the reflected $\sigma$-stable process of Guan and Ma \cite{GM}.  

We stress  that the boundary value problem \eqref{solutions0} is interpreted  in the sense of viscosity solutions, meaning that the Neumann boundary condition could also be not attained (and in this case the equation holds up to the boundary).

In the case of linear PIDES, as considered in \cite{BCGJ}, the kind of singularity of  $\mu$ influences the nature of the boundary value problem \eqref{solutions0}, in the sense that   the Neumann boundary condition is attained only
if the measure is singular enough. In particular in \cite{BCGJ} it is shown that, when the singularity is of order strictly less than $1$ as in \eqref{muintro},
 the equation holds up to the boundary and the process never reaches the boundary.
On the other hand, when the singularity of the measure is strong,  i.e. when $\mu$  is  of the type \eqref{muintro} with $\sigma \in [1,2)$, the situation is far more complicated, mainly due to the \textquotedblleft ugly\textquotedblright \, dependence in $x$ of the operator in \eqref{idol}  and to the interplay between the singularity of the measure and the geometry of the boundary. In \cite{BCGJ} this difficulty is tackled by considering  solutions which are in some sense H\"{o}lder continuous up to the boundary and  the comparison principle    is established only in this class. Though the result could not be optimal, it is  consistent with  the \textquotedblleft natural\textquotedblright\,  Neumann boundary condition   for the reflected $\sigma$-stable process (proved by Guann and Ma \cite{GM} through the variational
formulation and Green type formulas) which in the case of the half space reads
\begin{equation}\label{naturalnbc}
\lim_{t\to 0}t^{2-\sigma}\frac{\partial u}{\partial x_N}(x+te_N)=0.
\end{equation}
This allows  the normal derivative to growth less than $|x_N|^{\sigma-2}$ and then suggests that it is appropriate to look for solutions which are $\beta$-H\"{o}lder continuous, with $\beta>\sigma-1$, as assumed in \cite{BCGJ}. 
We remark  that the previous argument  suggests also  that, on the contrary, in the case $\sigma <1$ there is no need to assume any further regularity.

The situation is different when dealing with nonlinear equations as \eqref{solutions0}, which we consider in this paper. Indeed, the presence of the Hamiltonian term $H$ in \eqref{solutions0} entails further difficulties even in the case of measures of order strictly less than $1$ (e.g. as in \eqref{muintro}). This is due to the fact that  the nonlinear term $H$ could force the process to hit the boundary and, consequently, the  Neumann boundary condition to be attained. 

In order to deal with this difficulty, we consider  a class of Hamiltonians with a gradient growth stronger than the diffusive term in the nonlocal operator \eqref{idol}. The first example are Hamiltonians $H$ with   \textit{superfractional coercive} growth in the gradient variable, namely 
\begin{equation}\label{coer}
H(x,p)=a(x)|p|^m -f(x), 
\end{equation}
where $m>\sigma$, $\sigma \in (0,1)$,
 $a, f \, : \bar{\Omega} \mapsto \mathbb{R}$  are bounded and continuous functions and  $a(x) \geq a_0>0$ for some fixed constant $a_0$. We remark that the positivity of $a$ and the condition $m>\sigma$ make
the first-order term the leading term in the equation. We also observe that we
have no other additional restriction to $m$ (in particular, we can deal with Hamiltonians as in \eqref{coer} with
$m < 1$), allowing the study of Hamiltonians which are concave in
$Du$.

The second main example are Hamiltonians $H$ of \textit{Bellman type}, which arises in the study of Hamilton-Jacobi equations
associated to optimal exit time problems, such as
\begin{equation}\label{control}
H(x, p) = \sup_{\alpha \in \mathcal{A}} \{-b(x,\alpha) \cdot p - l(x,\alpha)\},
\end{equation}
where $\mathcal{A}$ is a compact metric space (the control space) and $b,l$ are continuous and bounded  functions (we refer the reader to  \cite{BCD} and \cite{FS} for some connections between this type of equations and control problems). 
Note that the diffusive term of $\mathcal{I}$  defined in \eqref{idol}
is of weaker order than the first-order term when we assume $\sigma < 1$. We also observe that, as in \cite{GT1} and \cite{T1}, the well-posedness of \eqref{solutions0} with  Hamiltonians as in \eqref{control} is
based on a careful study of the effects of the drift $b$ at each point of $\partial \Omega \times (0,+\infty)$. Note that the drift $b$ is not allowed to be parallel to the boundary for all controls, differently from the local case. This assumption on $b$  makes the Hamiltonian the leading order term in the equation and thus allows us to control the growth of  the  nonlocal term. 

The main result of our paper is the comparison principle between  bounded sub and super-viscosity solutions 
to \eqref{solutions0}, see Theorem \ref{thm}. We remark that the proof of this result is not standard even in the case $\sigma <1$ in the half space. The difficulties are mainly due to the fact that operators as in \eqref{idol} behave badly in $x$. The main idea which is behind the proof is to localize the argument on points which have the same distance from the boundary and this is carried out through the use of a non-standard non regular test function. The main assumption which allows us to localize on equidistant points is the superfractional growth of the Hamiltonian term, see in particular the proof of Proposition \ref{lembell} and Proposition \ref{lembellc} (more precisely,  Lemma \ref{lemdxdy} and Lemma \ref{lemdxdycoer}) for  Bellman and coercive Hamiltonians respectively. 

After the localization procedure, the rest of the proof in the case of the half space is simple, whereas in the case of general domains, further technical difficulties arise form the way the $x$-depending set of integration of $\mathcal{I}$ interferes with the geometry of the boundary. To face  these extra technical difficulties, we  rectify the boundary relying on the smoothness of $\Omega$. This is done in Lemma \ref{claimlip}  which is a key result used in the proof of Theorem \ref{thm}, which we prove before Theorem \ref{thm} in Section \ref{sec:claimlip}.

The first main application of our result is the proof of existence and uniqueness for \eqref{solutions0},  by standard Perron's method (Corollary \ref{corthm}).

Then, in Section \ref{appl}, we present some applications of our results to the evolutive setting. In particular, we prove the well-posedness of the Cauchy problem associated to \eqref{solutions0} and we study two different kind of asymptotic behaviour  under suitable assumptions on the data. We refer to Section \ref{appl} for precise assumptions, statement of the results and proofs.  

Finally, we remark that in \cite{GT1}  Dirichlet boundary value problems associated to nonlocal Hamilton-Jacobi parabolic equations are studied. The main interest of \cite{GT1} is to prove the well-posedness in the context of loss of the boundary condition, that is, in the case of Hamiltonians with  dominating gradient terms, as the ones considered in the present paper.  However, the  nonlocal operators treated in \cite{GT1}  are very different from the  censored operators  \eqref{idol} that we study here, since they are defined in all the space (and can therefore be seen as generalizations of the fractional laplacian). Moreover, differently from \eqref{solutions0},  in \cite{GT1} also boundary conditions  outside the domain are requested. This is strongly connected to the fact that here we consider Neumann boundary value problems with censored operators, that is, roughly speaking, we do not allow the process to jump outside the domain and therefore, no conditions outside is needed.
\subsection{Organization of the paper}
In Section \ref{sec:ass} we state the assumptions on the nonlocal operator and the Hamiltonian and we give the definition of solution to problem \eqref{solutions0}. In Section \ref{sec:proof} we state the main results, that is the uniqueness and existence for problem \eqref{solutions0} for a Hamiltonian either coercive or of Bellman type (Theorem \ref{thm} and Corollary \ref{corthm}). In Section \ref{sec:claimlip} we prove Lemma \ref{claimlip} and  in Section \ref{compstaz} we prove Theorem \ref{thm}. In Section \ref{appl} we treat the associated evolutive  problem, studying uniqueness, existence and asymptotic behaviour of the associated evolutive problem for large time. Finally, in the Appendix we prove some  lemmas used in the proof of Theorem \ref{thm}.

\section{Assumptions and  definition of solutions}\label{sec:ass}
We consider $\Omega \subset \mathbb{R}^N$  such that
\begin{equation}\tag{O}
\Omega \mbox{ is of class } W^{2,\infty}.
\end{equation}
This means that for any  $\hat{s} \in \partial \Omega$ there exists  $r=r(\hat{s})$ and a $W^{2,\infty}$-diffeomorphism $\psi: B_r(\hat{s}) \mapsto \mathbb{R}^N$
satisfying
$
 \psi_N(s)=d(s) \, \mbox{ for any } s \in B_r(\hat{s}),
$
where $d$ is the signed distance from the boundary of $\Omega$.
\begin{oss}\label{distance}
\rm{
By assumption (O), there exists a neighbourhood of the boundary of $\Omega$ where the distance from the boundary $d$ is smooth. Unless otherwise specified, throughout the paper we denote by  $d$ a function which coincides with the signed distance from the boundary of $\partial \Omega$ in this neighbourhood and is bounded in all the domain. We recall that by $n(x)$ we denote the exterior unit normal vector to $\partial \Omega$ and  we write $n(x)=-Dd(x)$ in the neighbourhood of the boundary where $d$ is smooth.
}
\end{oss}
We consider non-negative Radon measures with density $\frac{d\mu_x}{dz}$ satisfying
\begin{enumerate}
\item[(M)]  there exists $C_\mu, D_\mu>0, \sigma \in (0,1)$ such that for any $x,y \in \bar{\Omega},  z \in \mathbb{R}^N$ 
$$
\frac{d\mu_x}{dz}\leq C_\mu|z|^{-(N+\sigma)}, \quad |\frac{d\mu_x}{dz}-\frac{d\mu_y}{dz}|\leq D_\mu|x-y||z|^{-(N+\sigma)}.
$$
\end{enumerate}

For example, (M) is satisfied by
\begin{equation}\label{example}
d\mu_x=g(x,z)|z|^{-(N+\sigma)}dz \quad x \in \bar{\Omega}, z \in \mathbb{R}^N,
\end{equation}
where $\sigma \in (0,1)$, $g\, : \, \mathbb{R}^N \times \mathbb{R}^N \mapsto \mathbb{R}$ is a non-negative bounded function such that 
$g(\cdot,z)$ is Lipschitz  uniformly with respect to $z$.
 
Concerning  the jump function $j$ we assume 
\begin{enumerate}
\item[(J0)] for any  $x \in \bar{\Omega}$,
$
j(x,\cdot)\in C^1(\R^N),
$
$j(x,\cdot)$ is invertible,
$
j^{-1}(x,\cdot) \in C^{1}(\mathbb{R}^N)$ and there exists a constant $A_j$ such that
$$
|Dj^{-1}(x,\cdot)|\leq A_j \quad \forall x \in \bar{\Omega};
$$
\item[(J1)]  there exist $\tilde{C}_j,C_j, D_j>0$ such that for any $x, y \in \bar{\Omega},  z \in \mathbb{R}^N $, it holds
$$
\tilde{C}_j|z|\leq |j(x,z)|\leq C_j|z|, \quad |j(x,z)-j(y,z)|\leq D_j|z||x-y|.
$$\label{J1}
\end{enumerate}
For example (J0), (J1) are satisfied for
$$
j(x,z)=f(x)z \quad x \in \bar{\Omega}, z \in \mathbb{R}^N,
$$
where $f \, : \mathbb{R}^N \mapsto \mathbb{R}$ is  Lipschitz and bounded.

\subsection{Hamiltonian of Bellman type}
Let $\mathcal{A}$ be a compact metric space, $b\, : \, \bar{\Omega} \times \mathcal{A} \to \mathbb{R}^N$ and $f\, : \, \bar{\Omega} \times \mathcal{A}\to \mathbb{R}$  be continuous and bounded functions. We say that $H$ is of \textit{Bellman type} if for $x \in \bar{\Omega}, p \in \mathbb{R}^N, H(x,p)$ can be written as
\begin{equation}\label{bellman0}
H(x,p)=\sup_{\alpha \in \mathcal{A}} \{-b(x,\alpha)\cdot p -l(x,\alpha)\},
\end{equation}
and  satisfies the assumptions below.
We assume also:
\begin{enumerate}
\item[(C)] \label{lcont} \textit{Uniform  continuity  of the cost $l$}:\\
There exists a modulus of continuity  $\omega_l$ such that 
$$
|l(x,\alpha)-l(y,\alpha)|\leq \omega_l(|x-y|) \quad \forall \alpha \in \mathcal{A}, \forall x, y \in \bar{\Omega};
$$
\item[(L)] \label{blip} \textit{Uniform Lipschitz continuity of the drift $b$}:\\
$$
(\exists C>0)\, (\forall \alpha \in \mathcal{A}) \, (\forall x,y \in \bar{\Omega}) \, \, :\, \, 
|b(x, \alpha)-b(y,\alpha)| \leq C|x-y|. 
$$
\end{enumerate}
We introduce the following notation
\begin{equation}\label{gammain}
\Gamma_{\mbox{\footnotesize{in}}}:=\{x \in \partial \Omega \, :  \, b(x,\alpha) \cdot n(x) < 0 \, \, \forall \alpha \in \mathcal{A}\},
\end{equation}
\begin{equation}\label{gammaout}
\Gamma_{\mbox{\footnotesize{out}}}:=\{x \in \partial \Omega\,  | \, b(x,\alpha) \cdot n(x) >0 \quad \forall \alpha \in \mathcal{A}\},
\end{equation}
\begin{equation}\label{gamma}
\Gamma:=\{x \in \partial \Omega\, |\,  \exists \alpha_1, \alpha_2 \in \mathcal{A} \mbox{ s. t. } 
 b(x,\alpha_1)\cdot n(x)<0, 
b(x,\alpha_2)\cdot n(x) >0 \}.
\end{equation}
Roughly speaking, $\Gamma_{\mbox{\footnotesize{in}}}$ and $\Gamma_{\mbox{\footnotesize{out}}}$ can be respectively understood  as the set of points where the drift term pushes inside  and outside $\Omega$ the trajectories. 
 
In order to avoid two completely different drifts behaviour for arbitrarily closed points, we assume that each of these subsets is uniformly away from the others, as encoded in the following assumption (B). For example, if $\partial \Omega$ is connected, then it consists in one piece belonging to one
of  $\Gamma_{\mbox{\footnotesize{in}}}$, $\Gamma_{\mbox{\footnotesize{out}}}$ and $\Gamma$; otherwise, we are able to deal with boundary with several
components of different types, precisely each one belonging to one between $\Gamma_{\mbox{\footnotesize{in}}}$, $\Gamma_{\mbox{\footnotesize{out}}}, \Gamma$.

The assumptions we do on these subsets are the following
\begin{equation}\tag{B}
\Gamma_{\mbox{\footnotesize{in}}} \cup \Gamma_{\mbox{\footnotesize{out}}}\cup \Gamma=\partial \Omega, \quad \Gamma_{\mbox{\footnotesize{in}}}, \Gamma_{\mbox{\footnotesize{out}}}, \Gamma \mbox{ are  unions of connected components of } \partial \Omega.
\end{equation}

\begin{oss}\rm{
Note that the strict sign in the definition of $\Gamma_{\mbox{\footnotesize{in}}}, \Gamma_{\mbox{\footnotesize{out}}} $ and $\Gamma$ is fundamental, since it makes the Hamiltonian the leading order term in the equation, allowing us to control the growth of  the  nonlocal term, which  is of order strictly less than $1$.}
\end{oss}
\begin{oss}\rm{
In order to treat the points of $\Gamma_{\footnotesize{\mbox{in}}}$, we use the existence of a blow-up supersolution exploding on the boundary. We follow the same approach of \cite{BCGJ}, where the existence of a blow-up supersolution is proved  for censored type operators (of order strictly less than $1$) when  the measure of integration satisfies specific assumptions  (in particular does not depend on $x$ and there exists at least one point where it is strictly positive). In this particular case it is shown in \cite{BCGJ} that the integral term computed on the blow-up supersolution does not explode at the boundary. This is not true anymore when considering more general measures as we consider in  (M).
In order to solve this difficulty, we assume the strict sign in the behaviour of the drift term on $\Gamma_{\footnotesize{\mbox{in}}}$, which allows us to control the growth on the boundary of the integral term computed on this blow-up supersolution.  We refer to the proof of Proposition \ref{lembell} and in particular to Lemma \ref{lemblowup} for further details.
}
\end{oss}

\subsection{Coercive Hamiltonian and Examples}\label{coercass}

We consider \textit{superfractional} coercive Hamiltonians:

\begin{enumerate}
\item[(H1)] Let $\sigma$ be as in (M). There exists  $m > \sigma, c_0 >0,D>0$ such that for all $x \in \bar{\Omega},p \in \mathbb{R}^N$
$$
H(x,p) \geq c_0 |p|^m -D.
$$
\end{enumerate}

We distinguish the case of sub or superlinear coercivity:

\textit{Sublinear coercivity}:
We say that $H$ is sublinearly coercive if it satisfies (H1) for $m \leq  1$ and the following continuity condition holds:

\begin{enumerate}
\item[(Ha)] \label{assh2}There exists a constant $C>0$ and modulus of continuity $\omega_1$ such that, for  all $x,y,q,p \in \mathbb{R}^N$, we have
$$
H(y,p)-H(x,q) \leq \omega_1(|x-y|)(1+|p|)+ C(|p-q|).
$$
\end{enumerate}


\textit{Superlinear coercivity}: We say that $H$ is superlinearly coercive if:

\begin{enumerate}
\item[(Hb)]\label{assh}
There exists $m> 1$,$A, \bar{C}>0$ such that for all $\mu \in (0,1), x,y,p \in \mathbb{R}^N$, we have
$$
H(x,p)-\mu H(x,\mu^{-1}p) \leq (1-\mu)\left(\bar{C}(1-m)|p|^m +A\right);
$$
\item[(Hc)]\label{omega} If $m$ is as in assumption (Hb), there exist  $C>0$ and a modulus of continuity $\omega_1$ such that, for  all $x,y, q,p \in \mathbb{R}^N$
$$
H(y,p)-H(x,q) \leq \omega_1(|x-y|)(1+|p|^m\vee|q|^m)+ C|p-q|(|p|^{m-1}\vee |q|^{m-1}).
$$
\end{enumerate}

\begin{oss}
Note that condition $(Hb)$ implies $(H1)$ for $m>1$.
\end{oss}

As it is classical in viscosity solution's theory, the comparison principle allows the application of Perron's method to conclude the existence of solutions. To this end, we introduce the following assumption, which will allows us to build sub and supersolutions:
\begin{enumerate}
\item[(E)]\label{HC0}
 There exists $H_R>0$ such that for any $p  \in \mathbb{R}^N, |p|\leq R$
$$||H(\cdot,p)||_\infty \leq H_R.
$$
\end{enumerate}

\vspace{0.1cm}
As a model example for sublinearly coercive Hamiltonians, we consider 
$$
H(x, p) = a_1(x)|p|^m + a_2(x)|p|^l -f(x),
$$
with $m\leq 1, a_1\geq a_0>0$ for all $x \in \Omega, l<m$ and $a_1, a_2, f\, : \, \bar{\Omega} \mapsto \mathbb{R}$ are continuous and bounded functions and $a_1,a_2$ are also Lipschitz continuous. 

As a model example for superlinearly coercive Hamiltonian, we consider
$$
H(x, p) =a_1(x)|p|^m + a_2(x)|p|^l + b(x)\cdot p - f(x),
$$
with $m>1, l \leq m,  b$ bounded and continuous and $a_1,a_2, f$ as before. 

These Hamiltonians are coercive in $p$ and in the case $m>1$ we can
include transport terms with a Lipschitz continuous vector field $b\,: \, \bar{\Omega} \mapsto \mathbb{R}^N$.
The above assumptions are easily checkable in both cases.

\subsection{Notion of viscosity solutions}

We recall now the definition of solution to problem \eqref{solutions0}.  We use the following notation:
\begin{equation}\label{idoxi2}
\mathcal{I}[\phi]=\mathcal{I}_\xi[\phi]+\mathcal{I}^\xi[\phi],
\end{equation}
where
\begin{equation}\label{idoxi}
\mathcal{I}^\xi[\phi]=\int_{\footnotesize{\begin{matrix}|z|\geq \xi,\\ x+j(x,z) \in \bar{\Omega} \end{matrix}}} \phi(x+j(x,z))-\phi(x)d\mu_x(z).
\end{equation}
The $\mathcal{I}^\xi$-term is well-defined for any bounded function $\phi$. The $\mathcal{I}_\xi$-term is well-defined for $\phi \in C^1$ thanks to (M0).

We also denote
$$
F(x,u, Du,\mathcal{I}[u])=u(x)-\mathcal{I}[u](x)+H(x,Du).
$$
Following the approach of \cite{BCGJ}, we give the definition of viscosity solution to $\eqref{solutions0}.$ Let $C_j$ be defined as in (J1). 

\begin{defn}\label{defsol}
\begin{enumerate}
\item[(i)] A bounded usc function $u$ is a viscosity \textit{subsolution} to \eqref{solutions0} if, for any test-function $\phi \in C^1(\mathbb{R}^N) $ and maximum point $x$ of $u-\phi$ in $\bar{B}_{C_j\xi}(x)\cap \bar{\Omega}$
\begin{eqnarray*}
F(x, u(x), D\phi(x), \mathcal{I}_\xi[\phi]+\mathcal{I}^\xi[u]) \leq 0 \quad &x& \in \Omega \\
\min\{F(x,u(x), D\phi(x), \mathcal{I}_\xi[\phi]+\mathcal{I}^\xi[u]),\frac{\partial \phi}{\partial n}\}\leq 0 \quad &x& \in \partial\Omega.
\end{eqnarray*}
\item[(ii)]
A bounded lsc function $v$ is a viscosity \textit{supersolution} to \eqref{solutions0} if, for any test-function $\phi \in C^1(\mathbb{R}^N) $ and minimum point $x$ of $v-\phi$ in $\bar{B}_{C_j\xi}(x)\cap \bar{\Omega},$
\begin{eqnarray*}
F(x, v(x), D\phi(x), \mathcal{I}_\xi[\phi]+\mathcal{I}^\xi[v]) \geq 0 \quad &x& \in \Omega \\
\max\{F(x,v(x), D\phi(x), \mathcal{I}_\xi[\phi]+\mathcal{I}^\xi[v], \frac{\partial \phi}{\partial n}\}\geq 0 \quad &x& \in  \partial\Omega.
\end{eqnarray*}
\item[(iii)]
A viscosity \textit{solution} is both a sub- and a supersolution.
\end{enumerate}
\end{defn}

\section{Main results}\label{sec:proof}

The main result of this part is the following comparison principle for the problem \eqref{solutions0}. 

\begin{thm}\label{thm}\rm{[Comparison]}
\textit{
Let $\Omega$ be an open subset of $\mathbb{R}^N$ satisfying $(O)$. Assume $(M), (J0), (J1)$. Let
$H$ be  a Hamiltonian of Bellman type  as in \eqref{bellman0} satisfying (C), (L), (B) or  a coercive Hamiltonian satisfying (H1), (Ha)  or (H1), (Hb), (Hc).
Let $u$ be a bounded usc subsolution of \eqref{solutions0}  and $v$ a bounded lsc supersolution of \eqref{solutions0}. Then $u \leq v$ in $\bar{\Omega}$.}
\end{thm}

Once the comparison holds, we use the Perron's method for integro-differential equations (see \cite{AT}, \cite{BI}, \cite{SAY} and \cite{US},\cite{ISHP} for an introduction on the method) to get as a corollary existence and uniqueness for the problem \eqref{solutions0} either when $H$ is of  Bellman type either for $H$ coercive. 
  
\begin{cor}\label{corthm}\rm{[Existence and Uniqueness]}
\textit{
Let $\Omega$ be an open subset of $\mathbb{R}^N$ satisfying $(O)$. Assume $(M),(J0), (J1)$- Let $H$ be  either  a Hamiltonian of Bellman type  as in \eqref{bellman0} satisfying (C), (L), (B) or  a coercive Hamiltonian satisfying (H1), (Ha)  or (H1), (Hb), (Hc). 
 Assume $(E)$.
Then, there exists a unique bounded viscosity solution to problem \eqref{solutions0}.}
\end{cor}

\section{A preliminary key lemma}\label{sec:claimlip}
We prove the following Lemma \ref{claimlip}, which is a key result used in the proof of Theorem \ref{thm}.
Roughly speaking, it deals with the difficulties arising from the way the geometry of the boundary interferes with the singularity of the nonlocal terms. The scope is to estimate the nonlocal terms defined in \eqref{i1eps0} on points near the boundary and equidistant from it. The approach of  the proof is essentially  based on a rectification of the boundary, relying on its regularity.
\begin{oss}\rm{
In the case of domains with flat boundary, we do not need Lemma \ref{claimlip} in the proof of Theorem \ref{thm} since the estimation of the nonlocal terms can be carried out more easily. We refer to Remark \ref{flat}, step $4$ of the proof of Theorem \ref{thm}.}
\end{oss}

For any  $\hat{s} \in \partial \Omega$, let  $r=r(\hat{s})$ and $\psi: B_r(\hat{s}) \mapsto \mathbb{R}^N$  a $W^{2,\infty}$-diffeomorphism
as defined in assumption (O).
For $s_1, s_2 \in B_{\frac{r}{2}}(\hat{s})\cap \bar{\Omega}$, let
\begin{equation}\label{i1eps0}
\mathcal{I}[J_{s_1}/J_{s_2}]=\int_{\footnotesize{\begin{matrix}\small{J_{s_1}\setminus J_{s_2}},\\ \small{|z|\leq \d_0}\end{matrix}}}\frac{dz}{|z|^{N+\sigma-1}}
\end{equation}
where 
$
J_{s}=\{z \in \R^N \, |\, s+j(s,z) \in \bar{\Omega}\},
$
 $j$ satisfies assumptions (J0),(J1), $\sigma \in (0,1)$ and $0<\d_0<rC_j^{-1}/2$, where $C_j$ is the constant defined in (J1).

\begin{lem}\label{claimlip}
Let $\mathcal{I}[J_{\cdot}/J_{\cdot}]$  as in \eqref{i1eps0}   and assume $j$ satisfies $(J0), (J1)$. Let $\hat{s} \in \partial \Omega$, $r$ given as above and $s_1, s_2$  satisfying
\begin{equation}\label{ipolem}
d(s_1)=d(s_2), \quad s_1, s_2 \in B_{\frac{r}{2}}(\hat{s})\cap \bar{\Omega}. 
\end{equation}
Then there exists a positive constant $C$  such that
\begin{equation}\label{claim}
\mathcal{I}[J_{s_1}/J_{s_2}]\leq C|s_1-s_2|.
\end{equation}
\end{lem}

\begin{proofwb}
\begin{step1}-{\textit{Rectification of the boundary.}}
\upshape
We observe that since $s_1,s_2 \in B_{\frac{r}{2}}(\hat{s})\cap \bar{\Omega}, \d_0<rC_j^{-1}/2$ and by (J1), we have for any $|z|\leq \d_0$
\begin{equation}\label{d0}
s_1+j(s_1,z),\, s_2+j(s_2,z)\in B_r(\hat{s}). 
\end{equation}
By assumption (O), we  describe the domain of integration of $\mathcal{I}[J_{s_1}/J_{s_2}]$ through the diffeomorphism $\psi$   as follows
\begin{equation}\label{firsts}
s_1+j(s_1,z)\in \bar{\Omega}\Leftrightarrow\psi_N(s_1+j(s_1,z)) \geq 0, 
\end{equation}
\begin{equation}\label{seconds}
 s_2+j(s_2,z) \notin \bar{\Omega}\Leftrightarrow\psi_N(s_2+j(s_2,z))< 0.
\end{equation}
We observe that by assumption (O) and \eqref{ipolem}, we have
\begin{equation}\label{limitd}
\psi_N(s_1)=\psi_N(s_2).
\end{equation}
We proceed performing a change of variable in order to write the set of integration in terms of $\psi_N(s_1)$. 
In other words, we write
\begin{equation}\label{cofvclaimlip}
\psi(s_1+j(s_1,z))-\psi(s_1)=w,
\end{equation}
that is,
$
j(s_1,z)=\psi^{-1}\left(\psi(s_1)+w\right)-s_1.
$
By \eqref{cofvclaimlip}, we write \eqref{firsts} as follows
\begin{equation}\label{firstw}
w_N+\psi_N(s_1)\geq 0.
\end{equation}
Moreover, by (J1), (O) and since $\psi \in W^{2,\infty}$, we get
\begin{equation}\label{estzw}
\tilde{C}_j|z|\leq |j(s_1,z)|\leq |\psi^{-1}(\psi(s_1)+w-s_1)|\leq ||D\psi^{-1}|| _{\infty}|\psi(s_1)+w-\psi(s_1)|\leq ||D\psi^{-1}|| _{\infty} |w|.
\end{equation}
Similarly one can prove an analogous upper bound for $|w|$, getting
\begin{equation}\label{ref3}
\hat{C}|z|\leq |w|\leq \tilde{C} |z|,
\end{equation}
for some constants $\hat{C},\tilde{C}$ depending on the Lipschitz constants of $\psi, \psi^{-1}$ and on $\tilde{C}_j, C_j$ of (J1).

In the following step, we rewrite the set of integration  of $\mathcal{I}[J_{s_1}/ J_{s_2}]$ in a different way using the change of variable \eqref{cofvclaimlip} and the previous estimates.
\end{step1}
\begin{step2}-{\textit{Rewriting the set of integration.}}
\upshape
 By \eqref{cofvclaimlip} and if $\psi_N(s_2+j(s_2,z))\leq0$, we have
\begin{eqnarray}\label{ast}
w_N+\psi_N(s_1)\nonumber &=&\psi_N(s_2+j(s_2,z))+(\psi_N(s_1+j(s_1,z))-\psi_N(s_2+j(s_2,z)))\\ &\leq&(\psi_N(s_1+j(s_1,z))-\psi_N(s_2+j(s_2,z))).
\end{eqnarray}
For convenience of notation, let for the moment
\begin{equation}\label{shortji}
s(t)=ts_2+(1-t)s_1, \quad \zeta(t)=tj(s_2,z)+(1-t)j(s_1,z).
\end{equation}
Note that 
$
s(0)+\zeta(0)=s_1+j(s_1,z)$, $s(1)+\zeta(1)=s_2+j(s_2,z).
$
Then, since $\psi \in W^{2,\infty}$ and by \eqref{ast} we  write
$$
w_N+\psi_N(s_1) \leq\int_0^1 D\psi_N(s(t)+\zeta(t)) \cdot (s_1+j(s_1,z)-(s_2+j(s_2,z)) dt=A_1+A_2,
$$
where
$$
A_1=\int_0^1 [D\psi_N(s(t)+\zeta(t))-D\psi_N(s(t))] \cdot (s_1+j(s_1,z)-(s_2+j(s_2,z)) dt,
$$
$$
A_2=\int_0^1 D\psi_N(s(t)) \cdot (s_1-s_2)+\int_0^1 D\psi_N(s(t)) \cdot (j(s_1,z)-j(s_2,z))dt.
$$
From now on we denote by $C$ any positive constant which may change from line to line. 
By definition of $\zeta(t)$
\begin{equation}\label{bdt}
|\zeta(t)|\leq 2C_j|z| \quad \mbox{for any } t\in [0,1].
\end{equation}
Then, by \eqref{bdt}, (J1) and \eqref{estzw} we get
\begin{equation*}
A_1 \leq C\int_0^1|\zeta(t)|(|s_1-s_2|+|j(s_1,z)-j(s_2,z)|)dt \leq C|w||s_1-s_2|.
\end{equation*}
Now we analyse $A_2$. Note that  by \eqref{shortji} and  \eqref{limitd}
$$
\int_0^1 D\psi_N(s(t)) \cdot (s_1-s_2)=\int_0^1 D\psi_N((ts_2+(1-t)s_1) \cdot (s_1-s_2)=\psi_N(s_1)-\psi_N(s_2)=0.
$$
Moreover, since $\psi \in W^{2,\infty}$, by (J1) and \eqref{estzw}
\begin{equation*}
\int_0^1 D\psi_N(s(t)) \cdot (j(s_1,z)-j(s_2,z)))dt\leq C|w||s_1-s_2|.
\end{equation*}
Then we have
$
A_2\leq C|w||s_1-s_2|.
$
We denote $
a=\psi_N(s_1)
$
and observe
$
a\geq 0.
$
By \eqref{firstw} and the previous arguments, we get
\begin{equation}\label{lastw}
 -a\leq w_N \leq -a +C|s_1-s_2||w|.
\end{equation}
Then, we perform the change of variable in $\mathcal{I}[J_{s_1}/J_{s_2}]$ by (J0), (J1) and using that $\psi \in W^{2,\infty}$ and by \eqref{lastw} and  \eqref{ref3}, we get 
\begin{equation}\label{intw}
\mathcal{I}[J_{s_1}/J_{s_2}]\leq \bar{C}\int_{D} \frac{dw}{|w|^{N+\sigma-1}},
\end{equation}
where
$$
D=\{w \in \mathbb{R}^N \, : \, -a\leq w_N \leq -a +C|s_1-s_2||w|,\, 0<|w| \leq \tilde{C}\d_0\}.
$$
and $\bar{C}, \tilde{C}, C$ are positive constants, which we suppose  $\bar{C}=\tilde{C}=C=1$, by no loss of generality and for simplicity of exposition.
\end{step2}
\begin{step3}-{\textit{Estimate on $\tilde{D}$.}}
\upshape
 We introduce the following notation: 
\begin{equation}\label{betad}
d=(1-|s_1-s_2|)^{-1}, \quad \beta=(1+|s_1-s_2|)^{-1}.
\end{equation}
Note that by the second assumption in \eqref{ipolem}, $|s_1-s_2|\leq r$. Without loss of generality we can suppose $r\leq \frac{1}{2}$, so that we have $|s_1-s_2|\leq 1/2$.
Then
\begin{equation}\label{beta}
2\geq d\geq 1, \quad 1\geq \beta \geq \frac{1}{2}.
\end{equation}
Note that, if $w \in \tilde{D}$, then
\begin{equation}\label{domain}
 -a\leq w_N \leq  -a+|s_1-s_2||w'|+|s_1-s_2||w_N|.
\end{equation}
We identify two cases, depending on the sign of  $-a+|s_1-s_2||w'|$ and we denote
$$
D_1=\left\{w'\, |\, -a+ |s_1-s_2||w'|\geq 0,|w'|\leq \d_0\right\},
$$
and 
$$
D_2=\left\{w'\, |\, -a+|s_1-s_2||w'|< 0, |w'|\leq \d_0\right\}.
$$
Observe that, if $w\in \tilde{D} \cap D_2$, then
$
-a+|s_1-s_2||w'|< 0 
$
 and \eqref{domain}  implies
$
w_N < 0
$
and in particular
$
-a\leq w_N\leq -\beta a+\beta |s_1-s_2||w'|< 0.
$
Otherwise, if  $w \in \tilde{D}\cap D_1$, then
$
-a+|s_1-s_2||w'|\geq 0
$
and  $w_N$ can assume both negative and positive values. In particular \eqref{domain} implies
$
-a\leq w_N\leq -da +d|s_1-s_2||w'|.
$
Note also that  $-da +d|s_1-s_2||w'|\geq 0$.\\
By all the previous observations, we  write
\begin{equation}\label{divvar}
\int_{\tilde{D}} \frac{dw}{|w|^{N+\sigma-1}}=\int_{\tilde{D}} \frac{dw_Ndw'}{(|w'|^2+|w_N|^2)^{\frac{N+\sigma-1}{2}}} \leq \mathcal{F}_1+\mathcal{F}_2,
\end{equation}
where
$$
\mathcal{F}_1=\int_{D_1}\int_{-a}^{-da +d|s_1-s_2||w'|} \frac{dw_Ndw'}{(|w'|^2+|w_N|^2)^{\frac{N+\sigma-1}{2}}},
$$
$$
\mathcal{F}_2=\int_{D_2}\int_{-a}^{-\beta a+\beta|s_1-s_2||w'|}\frac{dw_Ndw'}{(|w'|^2+|w_N|^2)^{\frac{N+\sigma-1}{2}}}.
$$
For $\mathcal{F}_1$, we use that 
$
\frac{1}{|w'|^2+|w_N|^2}\leq \frac{1}{|w'|^2}
$
and by Fubini's Theorem, we integrate in the $N$-variable and we get
\begin{equation}\label{f1f2}
\mathcal{F}_1\leq \int_{D_1}\int_{-a}^{-da +d|s_1-s_2||w'|} \frac{dw_Ndw'}{|w'|^{N+\sigma-1}}\leq \int_{D_1} \frac{-da +d|s_1-s_2||w'|+a}{|w'|^{N+\sigma-1}}dw'.
\end{equation}
By the first of \eqref{betad} and \eqref{beta} and since $da \geq 0$, we have
$
-da +d|s_1-s_2||w'|+a=-da|s_1-s_2|+d|s_1-s_2||w'|\leq 2|s_1-s_2||w'|.
$  Therefore
\begin{equation}\label{primaest}
 \mathcal{F}_1\leq d|s_1-s_2| \int_{D_1} \frac{dw'}{|w'|^{N+\sigma-2}}.
\end{equation}
From now on we denote by $C$ any positive constant which may change from line to line.
Note that, since $w' \in \mathbb{R}^{N-1}$ and $\sigma <1$, we have
\begin{equation}\label{intdef}
\int_{D_1}\frac{dw'}{|w'|^{N+\sigma-2}}\leq C.
\end{equation}
Then by the previous observations, we get
\begin{equation}\label{f1f2fin}
\mathcal{F}_1\leq C|s_1-s_2|.
\end{equation}
Now we analyse  $\mathcal{F}_2$.
For simplicity of notation, we denote
$$
\zeta(w')=\int_{-a}^{-\beta a+\beta|s_1-s_2||w'|} \frac{dw_N}{\left(|w'|^2+w_N^2\right)^{\frac{N+\sigma-1}{2}}}
$$
and then, by Fubini's Theorem, we have
\begin{equation}\label{int}
\mathcal{F}_2=\int_{D_2}\zeta(w')dw'.
\end{equation}
We split the domain as follows
\begin{equation}\label{twoints}
\int_{D_2}\zeta(w')dw=\int_{D_2\cap\{a\leq |w'|\}}\zeta(w')dw'+\int_{\small{D_2\cap \{a> |w'|\}}}\zeta(w')dw'.
\end{equation}
We estimate the first term by
\begin{equation}\label{B1}
\int_{\small{D_2\cap  \{a\leq |w'|\}}}\zeta(w')dw'\leq \int_{\small{D_2\cap \{a\leq |w'|\}}}\frac{-\beta a+\beta|s_1-s_2||w'|+a}{|w'|^{N+\sigma-1}}dw' \leq  C|s_1-s_2|,
\end{equation}
where in the first inequality we used that $
-\beta a+\beta|s_1-s_2||w'|+a\leq 2|w'||s_1-s_2|,
$  since $\beta \leq 1$ and $a\leq |w'|$, and in the second inequality we used \eqref{intdef}.

Take now the second term in \eqref{twoints}. Note that, if $a > |w'|$, by  \eqref{betad} and  \eqref{beta},  we have 
$
-\beta a+\beta |s_1-s_2||w'|\leq-\beta a d^{-1}\leq -a4^{-1}\leq 0.
$
By all the previous observations, since the function
$
w_N \mapsto \frac{1}{\left(|w'|+w_N^2\right)^{\frac{N+\sigma-1}{2}}}
$
is increasing on the negative half line, we have
\begin{equation}\label{zeta1}
\zeta(w')\leq \frac{|s_1-s_2|(a+|w'|)}{\left(|w'|^2+4^{-2}a^2\right)^{\frac{N+\sigma-1}{2}}} \leq 2^{N+\sigma-1}|s_1-s_2|\frac{a+|w'|}{\left(|w'|^2+a^2\right)^{\frac{N+\sigma-1}{2}}}.
\end{equation}
Then 
\begin{equation}\label{c12}
\int_{\small{D_2\cap\{|w'|\leq a\}}}\frac{a+|w'|}{\left(|w'|^2+a^2\right)^{\frac{N+\sigma-1}{2}}}\leq 2a
\int_{D_2}\frac{dw'}{|(w',a)|^{N+\sigma-2}}\leq C\int_{D_2}\frac{dw'}{|w'|^{N+\sigma-2}}\leq C
\end{equation}
and coupling \eqref{zeta1} and \eqref{c12}, we get
\begin{equation}\label{B2}
\int_{\small{D_2\cap \{a\geq |w'|\}}}\zeta(w')dw'\leq C|s_1-s_2|.
\end{equation}
Then coupling \eqref{int}, \eqref{twoints}, \eqref{B1} and \eqref{B2}, we obtain
\begin{equation}\label{Bfin}
\mathcal{F}_2\leq C|s_1-s_2|
\end{equation}
and we conclude the proof by coupling \eqref{intw}, \eqref{divvar},  \eqref{f1f2fin}  and \eqref{Bfin}.\quad \quad \quad \quad $\Box$
\end{step3}
\end{proofwb}

\section{Proof of the comparison principle}\label{compstaz}
We prove Theorem \ref{thm} and we split the proof into two parts, depending whether $H$ is of Bellman type or coercive. 

\subsection{Hamiltonians of Bellman type}
The proof of Theorem \ref{thm} follows mainly by the following proposition, which we prove first. At the end of the proof of Proposition \ref{lembell}, we will prove Theorem \ref{thm}.

\begin{prop}\label{lembell}
Let $\Omega$ be an open subset of $\mathbb{R}^N$ satisfying $(O)$. Let \,$\mathcal{I}$ as in \eqref{idol}  and assume $\mu$ satisfies $(M)$, $j$ satisfies $(J0), (J1)$. Let $H$ be a Hamiltonian of Bellman type  as in \eqref{bellman0} satisfying (C), (L), (B) and let $u,v$ be respectively bounded  sub and supersolutions to \eqref{solutions0}.  
Then the function
$$
\omega(x):= u(x)-v(x)
$$
satisfies, in the viscosity sense, the equation
\begin{equation}\label{res}
\left\{
 \begin{array}{ll}
 \omega-\mathcal{I}[\omega](x)- B|D\omega|\leq 0 &\mbox{in} \, \,  \Omega, \\
                 \frac{\partial \omega}{\partial n}\leq 0 \, \, &\mbox{on} \, \, \partial \Omega,\\
                 \end{array}
\right.\,
\end{equation}
where $B$ is a positive constant depending on the data.
\end{prop}

\begin{prooflbwb}
Let $x_0\in \bar{\Omega}$ and $\phi \in C^1(\mathbb{R}^N)$ such that $\omega-\phi$ has a  maximum point  at $x_0$ in $\bar{B}_{C_j \xi}(x_0) \cap \bar \Omega$ for some $0<\xi<1$. Note that the restriction $\xi<1$ is only for simplicity of exposition. We can suppose that $x_0$ is a strict maximum point by considering $\tilde{\phi}=\phi(x)+ k|x-x_0|^2$ for $k$ sufficiently large. In particular, note that we have the following property
$$
\lim_{\eps\to 0^+}\sup_{|x-x_0|\leq \eps} |D\tilde{\phi}(x)-D\phi(x_0)|=0.
$$
From now on, we consider the test function $\tilde{\phi}$ and we still denote it by $\phi$, with some abuse of notation.
We  observe that if $x_0 \in \Omega$ the proof is rather standard, since in this case the maximum points $(x,y)$ of $u-v-\phi$ converge as $\eps \to 0$ to $(x_0,x_0)$ and hence they are bounded away from the boundary for $\eps$ small enough. This last property implies that we can directly use the equations  and then proceed as in the following case.

Let $\Gamma_{\footnotesize{\mbox{in}}}, \Gamma_{\footnotesize{\mbox{out}}}, \Gamma$ be defined respectively in \eqref{gammain}, \eqref{gammaout} and \eqref{gamma} and recall they satisfy (B). We suppose $x_0 \in \partial \Omega$ and we split the proof depending if 
\begin{itemize}
\item[(a)]
$
x_0\in \Gamma_{\footnotesize{\mbox{in}}};
$
\item[(b)] 
$x_0 \in \Gamma_{\footnotesize{\mbox{out}}};
$
\item[(c)] 
$x_0\in \Gamma.
$ 
\end{itemize}
In  case $(a)$  we use the existence of the blow-up supersolution  which explodes at the boundary  and allows us to keep the maximum points far from the boundary. 
Since the proof in this case  is inspired by a similar approach used in \cite{BCGJ}, we give the details at the end of the proof of $(b)$ and $(c)$ in Remark \ref{proofa}.
Now we treat case $(b)$ and $(c)$. Since the proofs are similar, we treat them at the same time.

We suppose that
\begin{equation}\label{hpnorm}
\frac{\partial \phi}{\partial n}(x_0)>0,
\end{equation}
and prove that the F-viscosity inequality of Definition \eqref{defsol} hold for $\omega$.
\begin{step1}{Localising on equidistant points (that is, $d(x)=d(y)$).}
\upshape
Let $\eps>0$ and  $d$ be a function as  in Remark \ref{distance}. We double the variable by   introducing the function for $\eps, \d>0$
\begin{equation}\label{phid}
\tilde{\phi}(x,y)=\phi((x+y)/2)+\eps^{-1}\chi_\eps(|x-y|)+K\eps^{-1}\chi_\d(|d(x)-d(y)|), 
\end{equation}
where  $\chi_\eps: \, \mathbb{R}\to \mathbb{R}$ (and similarly $\chi_\d$) is defined as follows
\begin{equation}\label{chieps0}
\chi_\eps(r)=\sqrt{r^2+\eps^4}\quad r \in \mathbb{R}
\end{equation} 
and $K>0$ is a constant large enough such that
\begin{equation}\label{K}
K > (2+C_2)\gamma^{-1},
\end{equation}
where $\gamma, C_2>0$ depend on $x_0$ and are precisely  defined in Lemma \ref{lemh} in the Appendix (for $\hat{s}=x_0$).
Let
\begin{equation}\label{PHI2}
\Phi(x,y)= u(x)-v(y)-\tilde{\phi}(x,y)
\end{equation}
and denote by $(\bar{x}, \bar{y})$ the maximum point of $\Phi$ in $\bar{B}_{2C_j\xi'}(x_0)\cap \bar{\Omega} \times \bar{B}_{2C_j\xi'}(x_0)\cap \bar{\Omega}$ for $0<\xi'<\frac{\xi}{2}$. We observe that $(\bar{x}, \bar{y})$ depends now also on $\d$ and we omit the dependence. \\
Now consider 
\begin{equation}\label{PHI0}
\Psi(x,y)=u(x)-v(y)-\psi(x,y),
\end{equation}
where
\begin{equation}\label{testpsi0}
\psi(x,y)=\phi((x+y)/2)+\eps^{-1}\chi_\eps(|x-y|)+K\eps^{-1}|d(x)-d(y)|,
\end{equation}
where $d$ is the signed distance from the boundary (see Remark \ref{distance}), $\chi_\eps$ is defined as in \eqref{chieps0} and $K$ is as in \eqref{K}. Note that the test function in \eqref{testpsi0} is not differentiable on the points such that $d(x)=d(y)$.
By upper-continuity, $\Psi$ in \eqref{PHI0} attains its maximum over 
$$
A:=\bar{B}_{2C_j\xi'}(x_0)\cap \bar{\Omega}\times \bar{B}_{2C_j\xi'}(x_0)\cap \bar{\Omega}
$$
at a point $(x,y)$. By classical arguments in viscosity solution theory, we get as $\eps \to 0$
\begin{equation}\label{epszero}
x,y \to x_0, \quad \eps^{-1}\chi_\eps(|x-y|)\to 0, \quad \eps^{-1}|d(x)-d(y)|\to 0 
\end{equation}
and
$
u(x)-v(y)-\tilde{\psi}(x,y)\to u(x_0)-v(x_0)-\phi(x_0).
$
Note also that by the second of \eqref{epszero}, it directly follows
\begin{equation}\label{epszerodir}
\eps^{-1}|x-y|\to 0 \mbox{ as } \eps \to 0
\end{equation}
We prove the following key lemma.
\begin{lem}\label{lemdxdy}
Under the above notation, we have
\begin{itemize}
\item[(i)] $\bar{x} \to x, \, \,\bar{y} \to y, \,\, u(\bar{x})\to u(x), \,\, v(\bar{y}) \to v(y)\,\, \mbox{ as } \d \to 0;$
\item[(ii)]
$d(x)=d(y)$;
\end{itemize}
\end{lem}
\begin{proof}
Note that (i)  follows by classical argument in viscosity solution theory. 
We remark that the proof of (ii) is slightly  different in  case $(b)$ and  case $(c)$. 
We argue by contradiction and we suppose that $d(x)\neq d(y)$. 
  First we prove that the $F$-viscosity inequalities for $u$ and $v$ of Definition \eqref{defsol} hold.
Suppose that $x \in \partial \Omega$, then $d(x)=0$ and $d(y)\neq0$. 
We denote
\begin{equation}\label{hatp}
\hat{p}=\frac{x-y}{|x-y|}
\end{equation}
and we write
\begin{equation}\label{derphiuno}
\frac{\partial \psi}{\partial n}(\cdot,y)(x)=\frac{1}{2}\frac{\partial \phi}{\partial n}((x+y)/2)+\eps^{-1}\chi'_\eps(|x-y|)\hat{p} \cdot n(x) + K\eps^{-1}.
\end{equation}
 Note that
\begin{equation}\label{gradchi}
0\leq \chi'_\eps(|x-y|)\leq 1.
\end{equation}
Note that by \eqref{epszero}, we can suppose that $x,y$ are close to the boundary, by taking $\eps$ small enough. By the Taylor's formula for the distance function, we have for $\eps$ small enough
$$
n(x)\cdot (x-y)+\frac{1}{2}(x-y)^TD^2d(x)(x-y)+ o(|x-y|^2)=d(y)\geq 0
$$
and then
\begin{equation}\label{taylor}
n(x)\cdot(x-y)\geq-||D^2d||_\infty|x-y|^2/2+o(|x-y|^2).
\end{equation}
By \eqref{hatp}, \eqref{gradchi}, \eqref{taylor} and \eqref{epszero}, we have
\begin{equation}\label{taylorgrad}
\eps^{-1}\chi'_\eps(|x-y|)\hat{p}\cdot n(x)\geq o_\eps(1).
\end{equation}
Note that, from \eqref{hpnorm}, for $\eps$ small enough we have also
\begin{equation}\label{phider0}
\frac{\partial \phi}{\partial n}((x+y)/2)>\frac{1}{2}\frac{\partial \phi}{\partial n}(x_0)>0.
\end{equation}
By \eqref{derphiuno}, \eqref{phider0}, \eqref{taylorgrad} and since $K\geq0$,  we conclude for  $\eps$ small enough 
\begin{equation}\label{aubord}
\frac{\partial \psi}{\partial n}(\cdot, y)(x)\geq \frac{1}{4}\frac{\partial \phi}{\partial n}(x_0)+o_\eps(1)+K\eps^{-1}>0.
\end{equation}
Then, since $u$ is a viscosity subsolution and the function $u(\cdot)-v(y)-\psi(\cdot,y)$ has a local maximum at $x$, the $F$-viscosity inequality  of Definition \eqref{defsol}(i) holds. A similar argument can be carried out for $v$.
From now on, we treat separately  \textit{Case $(b)$} ($x_0\in \Gamma_{\mbox{\footnotesize{out}}}$) and \textit{Case $(c)$} ($x_0\in \Gamma$).

\textit{Case $(b)$} \quad In this case $x_0\in \Gamma_{\mbox{\footnotesize{out}}}$, where $\Gamma_{\mbox{\footnotesize{out}}}$ is defined in \eqref{gammaout}.  Suppose $d(x)>d(y)$.  Then, by Definition \eqref{defsol}(i) and by \eqref{aubord}, we have 
\begin{eqnarray}\label{equ0}
u(x)-\mathcal{I}_{\xi'}[\psi(\cdot,y)](x)-\mathcal{I}^{\xi'}[u](x) +H(x,D[\psi(\cdot,y)](x))\leq 0.
\end{eqnarray}
Note that
$$
D[\psi(\cdot,y)](x)=\eps^{-1}(\chi'_\eps(|x-y|)\hat{p}-Kn(x)) +q,
$$
where $\hat{p}$ is defined in \eqref{hatp} and
 \begin{equation}\label{qxqy}
q=D\phi((x+y)/2)/2. 
\end{equation} 
We apply Lemma \ref{lemh}, \eqref{a1} with $\hat{s}=x_0$, $p=\eps^{-1}\chi'_\eps(|x-y|)\hat{p}+q$  and $\lambda=\eps^{-1}K$ and by the definition \eqref{K} of $K$, we get for $\eps$ small
\begin{eqnarray}\label{A.1}
H(x,D[\psi(\cdot,y)](x)) \nonumber &\geq& \eps^{-1}\gamma K-C_2\left|\eps^{-1}\chi'_\eps(|x-y|)\hat{p}+q\right|-C_2\\ \nonumber& \geq &\eps^{-1}(\gamma K-C_2)-C\\ &\geq & \eps^{-1}-C,
\end{eqnarray}
where by $C$, here and in the following, we denote any positive constant independent of $\eps$ which may change from line to line.
To estimate the nonlocal terms we use the following lemma, which we prove in the Appendix.

\begin{lem}\label{estnonloclem}
Let $\mathcal{I}^\xi, \mathcal{I}_\xi$ be as in \eqref{idoxi}, \eqref{idoxi2} and assume  the first of $(M)$ and $ (J1)$. Under the above notation, for any $\xi >0$, there exist a positive constants $C_1$ independents of $\eps$ such that 
\begin{itemize}
\item[(i)]
$
-\mathcal{I}_\xi[\psi(\cdot, y)](x)-\mathcal{I}^\xi[u](x) \geq -\eps^{-1}C_1\xi^{1-\sigma}-C_1\xi^{-\sigma};
$
\item[(ii)]
$
 -\mathcal{I}_\xi[\psi(x, \cdot)](y)-\mathcal{I}^\xi[v](y) \leq \eps^{-1}C_1\xi^{1-\sigma}+C_1\xi^{-\sigma}.
$
\end{itemize}
\end{lem}

Then, by \eqref{A.1}, by Lemma \ref{estnonloclem} $(i)$ with $\xi'=\eps$ and by the boundedness of $u$, we write \eqref{equ0} as follows 
$$
-\eps^{-\sigma}+\eps^{-1} \leq C,
$$
and we reach a contradiction for $\eps $ small enough, since  $C$ is independent of $\eps$ and $\sigma <1$.\\
Now suppose $d(x) <d(y)$. In this case we use the following $F$-viscosity inequality for the supersolution  $v$ 
\begin{eqnarray}\label{eqv0}
v(y)-\mathcal{I}_{\xi'}[-\psi(x,\cdot)](y)-\mathcal{I}^{\xi'}[v](y) +H(y,-D[\psi(x,\cdot)](y))\geq 0.
\end{eqnarray}
We have
$$
D[-\psi(\cdot,y)](x)=\eps^{-1}(\chi'_\eps(|x-y|)\hat{p}+Kn(y)) -q,
$$
where $\hat{p}$ is defined in \eqref{hatp} and  $q$ is defined in \eqref{qxqy}.
Then , for $\eps$ small enough, we apply Lemma \ref{lemh}, \eqref{a2} with  $\hat{s}=x_0$, $p=\eps^{-1}\chi'_\eps(|x-y|)\hat{p} -q$ and $\lambda=\eps^{-1}K$  and by  \eqref{K} we get
\begin{eqnarray}\label{A.2}
H(y,-D[\psi(x,\cdot)](y)) \nonumber &\leq& -\eps^{-1}\gamma K+C_2\left|\eps^{-1}\chi'_\eps(|x-y|)\hat{p} -q\right|\\ \nonumber &\leq& -\eps^{-1}(\gamma K+C_2)+C\\&\leq& -\eps^{-1}-C.
\end{eqnarray}
We proceed as in the previous case, we apply Lemma \ref{estnonloclem} $(ii)$ with $\xi=\eps$ and by \eqref{A.2} and the boundedness of $v$ , we get
$$
\eps^{-\sigma}-\eps^{-1}\geq C
$$
and we reach a contradiction for $\eps$ small enough as above.

\textit{Case $(c)$} \quad In this case  $x_0\in \Gamma$, where $\Gamma$ is defined in \eqref{gamma}. If $d(x)>d(y)$ the proof is the same. If $d(x)<d(y)$ we write again equation \eqref{equ0} and since
$$
D[\psi(\cdot,y)](x)=\eps^{-1}(\chi'_\eps(|x-y|)\hat{p}+Kn(x)) +q,
$$
where $\hat{p}$ is defined in \eqref{hatp},
  we apply Lemma \ref{lemh} \eqref{a4} with  $\hat{s}=x_0$, $p=\eps^{-1}\chi'_\eps(|x-y|)\hat{p}+q$  and $\lambda=\eps^{-1}K$, for $\eps$ enough small, and we conclude as above.
\end{proof}
\end{step1}

\begin{step2}{Writing the viscosity inequalities.}
\upshape
By \eqref{epszero} and Lemma \ref{lemdxdy} (i), from now on we  consider $\d, \eps$ small enough so that 
\begin{equation}\label{epsdsmall}
\bar{x}, \bar{y}, x,y \in B_{C_j\xi'}(x_0) \cap \bar{\Omega}.
\end{equation}
Now we prove that the $F$-viscosity inequalities for $u$ and $v$ hold. We take $\bar{x} \in \partial \Omega$ and we show that the boundary conditions do not hold, so the $F$-viscosity inequalities hold as in Definition 	\eqref{defsol}.  We proceed exactly as in Step $1$,  Lemma \ref{lemdxdy}, so we omit the details. We recall that for all $\d>0$ 
\begin{equation}\label{chidder}
0\leq \chi'_\d(|x-y|)\leq 1 \quad \mbox{ for all } x,y\in \bar{\Omega},
\end{equation}
and we note only that since $d(\bar{x})=0$, we have for $\eps, \d $ small enough
$$
\frac{\partial \tilde{\phi}}{\partial n}(\cdot,\bar{y})(\bar{x})\geq \frac{1}{4}\frac{\partial \phi}{\partial n}(x_0)+K\eps^{-1}\chi_{\d}'(d(\bar{y})) +o_{\d,\eps}(1)>0,
$$
where $o_{\d,\eps}(1)$ means that $\lim_{\d \to 0}o_{\d,\eps}(1)=o_\eps(1)$.
Then  we have
\begin{eqnarray}\label{eq}
u(\bar{x})-v(\bar{y})\nonumber &\leq& H(\bar{y},-D[\tilde{\phi}(\bar{x},\cdot)](\bar{y}))-H(\bar{x},D[\tilde{\phi}(\cdot,\bar{y})](\bar{x})\\ &+&\mathcal{I}^{\xi'}[u](\bar{x})-\mathcal{I}^{\xi'}[v](\bar{y})+\mathcal{I}_{\xi'}[\tilde{\phi}(\cdot,y)](\bar{x}) -\mathcal{I}_{\xi'}[-\tilde{\phi}(\bar{x},\cdot)](\bar{y}).
\end{eqnarray}
Since $\tilde{\phi} \in C^1$, by (J1) and the first of (M), we have
\begin{equation}\label{oxi}
\mathcal{I}_{\xi'}[\tilde{\phi}(\cdot, \bar y)](\bar x) \leq C_j||D\tilde{\phi}||_{L^{\infty}(\bar{B}(0,C_j\xi'))}\int_{\mathbb{R}^n}1_{|z|\leq \xi'}|z|d\mu_{x}(z)=o_{\xi'}(1).
\end{equation}
where $C_j$ is as in (J1) and $o_{\xi'}(1)$ is independent of $\d$. The same holds for $-\mathcal{I}_{\xi'}[-\tilde{\phi}(\bar{x},\cdot)](\bar{y})$.

Note that 
\begin{equation}
\left|D[\tilde{\phi}(\cdot,\bar{y})](\bar{x})-D[-\tilde{\phi}(\bar{x},\cdot)](\bar{y})\right|=\eps^{-1}\left|K\chi'_\d(|d(\bar{x})-d(\bar{y})|)\tilde{p}\left(n(\bar{y})-n(\bar{x})\right)\right| +|D\phi((\bar{x}+\bar{y})/2)|,
\end{equation}
where
\begin{equation}\label{hatpdelta}
\tilde{p}=\frac{d(\bar{x})-d(\bar{y})}{|d(\bar{x})-d(\bar{y})|}.
\end{equation}
For $\d, \eps$ small enough, we suppose that $\bar{x},\bar{y}$  belong to the neighbourhood of the boundary where the distance is smooth.
By \eqref{chidder} and the smoothness of the distance function we have
\begin{eqnarray}\label{diff}
\left|D[\tilde{\phi}(\cdot,\bar{y})](\bar{x})-D[-\tilde{\phi}(\bar{x},\cdot)](\bar{y})\right|\nonumber &\leq & \eps^{-1}K|n(\bar{y})-n(\bar{x})|+|D\phi((\bar{x}+\bar{y})/2)|\\&\leq & \eps^{-1}K|\bar{x}-\bar{y}|+|D\phi((\bar{x}+\bar{y})/2)|.
\end{eqnarray}
By the definition of $H$ and \eqref{diff}, we have
\begin{equation}\label{primah}
H(\bar{y},-D[\tilde{\phi}(\bar{x},\cdot)](\bar{y})) -H(\bar{y},D[\tilde{\phi(}\cdot,\bar{y})](\bar{x}))\leq B\left(|D\phi((\bar{x}+\bar{y})/2)|+K\eps^{-1}|\bar{x}-\bar{y}|\right),
\end{equation}
where $B=\sup_{\footnotesize{x\in \bar{\Omega}, \alpha \in \mathcal{A}}}b(x,\alpha)$. Moreover by $(C), (L)$,  we have
\begin{equation*}
H(\bar{y},D[\tilde{\phi}(\cdot,\bar{y})](\bar{x})-H(\bar{x},D[\tilde{\phi}(\cdot,\bar{y})](\bar{x}))\\\nonumber\leq B|\bar{x}-\bar{y}||D[\tilde{\phi}(\cdot,\bar{y})](\bar{x})| +\omega_l(|\bar{x}-\bar{y}|)
\end{equation*}
and since
\begin{equation}\label{modgrad}
|D[\tilde{\phi}(\cdot,\bar{y})](\bar{x})|\leq K\eps^{-1}+\eps^{-1}|\bar{x}-\bar{y}|+2^{-1}||D\phi||_{L^\infty(B_{2C_j\xi'}(x_0))},
\end{equation}
we get
\begin{equation}\label{secondah}
H(\bar{y},D[\tilde{\phi}(\cdot,\bar{y})](\bar{x})-H(\bar{x},D[\tilde{\phi}(\cdot,\bar{y})](\bar{x}))\leq C\left(\eps^{-1}|\bar{x}-\bar{y}|+|\bar{x}-\bar{y}|\right)+\omega_l(|\bar{x}-\bar{y}|),
\end{equation}
where $C>0$ is a constant depending on $B, K$ and $||D\phi||_{L^\infty(B_{C_j}(x_0))}$. 
By coupling \eqref{primah}, \eqref{secondah}, \eqref{epszerodir} and (i) of Lemma \ref{lemdxdy}, we get
\begin{equation}\label{estham}
H(\bar{y},-D[\tilde{\phi}(\bar{x},\cdot)](\bar{y}))-H(\bar{x},D[\tilde{\phi}(\cdot,\bar{y})](\bar{x})\leq B|D\phi((\bar{x}+\bar{y})/2)|+o_{\d,  \eps}(1),
\end{equation}
where  $o_{\d, \eps}(1)$ means $\lim_{\d \to 0} o_{\d,  \eps}(1)=o_{\eps}(1)$. 
Plugging \eqref{estham} and \eqref{oxi} into \eqref{eq}, we get
\begin{equation}\label{eq0}
u(\bar{x})-v(\bar{y})\leq B|D\phi((\bar{x}+\bar{y})/2)|+\mathcal{I}^{\xi'}[u](\bar{x})-\mathcal{I}^{\xi'}[v](\bar{y})+o_{\d, \eps}(1)+o_{\xi'}(1).
\end{equation}
\end{step2}
\begin{step3}{Sending $\d \to 0$.}
\upshape
We want to send first $\d \to 0$ in \eqref{eq0} and we observe that the nonlocal terms are uniformly bounded in $\d$. 
Consider $\mathcal{I}^{\xi'}[u](\bar{x})$, observing that the same argument works similarly for  $\mathcal{I}^{\xi'}[v](\bar{y})$.
Note that by \eqref{epsdsmall} and (J1), if $|z|< 1$, then  $\bar{x}+j(\bar{x},z) \in B_{2C_j\xi'}(x_0)$. Since $(\bar{x},\bar{y})$ is a maximum point on $\bar{B}_{2C_j\xi'}(x_0)\cap \bar{\Omega} \times \bar{B}_{2C_j\xi'}(x_0)\cap \bar{\Omega}$ of $\Phi$ defined in \eqref{PHI2}, we have for $\d, \eps$ small
\begin{eqnarray}\label{maxv0}
u(\bar{x}+j(\bar{x},z))-u(\bar{x})\nonumber &=&u(\bar{x}+j(\bar{x},z))-v(\bar{y})-(u(\bar{x})-v(\bar{y}))  \\ &\leq& \tilde{\phi}(\bar{x}+j(\bar{x},z),\bar{y})-\tilde{\phi}(\bar{x},\bar{y}).
\end{eqnarray}
Note that   $\chi_\d$ is Lipschitz with Lipschitz constant independent of $\d$ thanks to \eqref{chidder}. Then, by the definition of $\tilde{\phi}$, since  $\chi_\eps, \chi_\d, \phi$ are Lipschitz and by (J1), we have
\begin{equation}\label{firstmax}
u(\bar{x}+j(\bar{x},z))-u(\bar{x})\leq C\eps^{-1}|z|+C|z|,
\end{equation}
which, by the first of (M), gives the uniform boundedness in $\d$ of $\mathcal{I}^{\xi'}[u](\bar{x})$ when $|z|<1$. 
When $|z|\geq 1$, the claim simply follows by  the boundedness of $u$ and the first of (M).\\
Then, we send $\d \to 0$ in \eqref{eq0} and, by the semicontinuity and boundedness of $u$ and $v$ and Lemma \ref{lemdxdy} (i), apply Fatou's Lemma and   we get
\begin{equation}\label{eq003}
u(x)-v(y)\leq B|D\phi((x+y)/2)|+\mathcal{I}^{\xi'}[u](x)-\mathcal{I}^{\xi'}[v](y)+o_{\eps}(1)+o_{\xi'}(1).
\end{equation}
In the next step, we estimate the nonlocal terms by bounds independent of the parameter $\xi'$ small. Indeed, at the end of the proof, we will send first $\xi'\to 0$ and then $\eps \to 0$.  Note also that now, thanks to Lemma \ref{lemdxdy} (ii), we have that $d(x)=d(y)$.
\end{step3}

\begin{step4}{Estimate of the nonlocal terms.}
\upshape
We prove the following lemma.
\begin{lem}\label{lemestnonlocb}
Under the above notation, we have
\begin{equation}\label{leminpiu}
\mathcal{I}^{\xi'}[u](x)-\mathcal{I}^{\xi'}[v](y)\leq  C\eps^{-1}|x-y|+\mathcal{P}_\xi+\mathcal{K}^{\xi}+o_{\eps}(1)+o_{\xi'}(1),\end{equation}
where $\mathcal{K}^\xi, \mathcal{P}_\xi$ are as in \eqref{kxi} and \eqref{pxidef}  and $C>0$ is independent of all the parameters.
\end{lem}
\begin{oss}
In the proof of Lemma \ref{lemestnonlocb}, we deeply rely on the assumption $\sigma \in (0,1)$.
\end{oss}
\begin{proof}
 For simplicity of exposition, we first conclude the proof when the measure $\mu$ in the nonlocal terms has no dependence on $x$, i.e. $\mu_x\equiv \mu$. We refer to Remark \ref{ossmux} for details in the case of $x$-dependence.
We write
\begin{equation}\label{123}
\mathcal{I}^{\xi'}[u](x)-\mathcal{I}^{\xi'}[v](y)= \mathcal{I}^{\xi'}[J_x/J_y]+\mathcal{I}^{\xi'}[J_y/J_x] +\mathcal{T}^{\xi'}[J_x\cap J_y], 
\end{equation}
where 
$
J_x=\{z\in \mathbb{R}^n \, | \, x+j(x,z) \in \bar{\Omega}\}
$
and
\begin{equation}\label{jxjy}
\mathcal{I}^{\xi'}[J_x/J_y]=\int_{\footnotesize{\begin{matrix} J_x/J_y, \\  |z|\geq \xi' \end{matrix}}} u(x+j(x,z))-u(x) d\mu(z),
\end{equation}
$$
\mathcal{I}^{\xi'}[J_y/J_x]= \int_{\footnotesize{\begin{matrix} J_y/J_x, \\ |z|\geq \xi'\end{matrix}}} v(y)-v(y+j(y,z)) d\mu(z),
$$
\begin{equation}\label{jxujy}
\mathcal{T}^{\xi'}[J_x\cap J_y]=\int_{\footnotesize{\begin{matrix}J_x\cap J_y, \\ |z|\geq \xi'\end{matrix}}} [u(x+j(x,z))-u(x)-(v(y+j(y,z))-v(y))]d\mu(z).
\end{equation}
Consider $\mathcal{T}^{\xi'}[J_x\cap J_y]$.  Recall that $(\bar{x},\bar{y})$ satisfy  for any $ x',y' \in \bar{B}_{2C_j\xi'}(x_0)\cap \bar{\Omega} \times \bar{B}_{2C_j\xi'}(x_0)\cap \bar{\Omega}$
\begin{equation}\label{maxprimad}
u(\bar{x})-v(\bar{y})-\tilde{\phi}(\bar{x},\bar{y})\geq u(x')-v(y')-\tilde{\phi}(x',y').
\end{equation}
Letting $\d \to 0$ in \eqref{maxprimad}, by (i) of Lemma \ref{lemdxdy}, the definition of $\tilde{\phi}$ and the semicontinuity of $u,v$, we get for any $ x',y' \in \bar{B}_{2C_j\xi'}(x_0)\cap \bar{\Omega} \times \bar{B}_{2C_j\xi'}(x_0)\cap \bar{\Omega}$
\begin{multline}\label{maxdopod0}
u(x')-u(x)-(v(y')-v(y))\leq \eps^{-1}\chi_\eps(|x'-y'|)-\eps^{-1}\chi_\eps(|x-y|)\\+\phi((x'+y'/2)
-\phi((x+y)/2).
\end{multline}
If  $|z|< 1$, then by \eqref{epsdsmall} and (J1), $x+j(x,z),y+j(y,z) \in B_{2C_j\xi'}(x_0)$. Then we write \eqref{maxdopod0} for $x'=x+j(x,z), y'=y+j(y,z)$ and we have
\begin{eqnarray*}
u(x+j(x,z))-u(x)\nonumber&-&(v(y+j(y,z))-v(y)) \\ \nonumber &\leq
& \eps^{-1}\chi_\eps(|x+j(x,z)-y-j(y,z)|)-\eps^{-1}\chi_\eps(|x-y|)\\ &+&\phi((x+j(x,z)+y+j(y,z))/2)
-\phi((x+y)/2).
\end{eqnarray*}
Note that by the Lipschitz continuity of $\chi_\eps$, (J1) and \eqref{epszero}, we have
\begin{equation}\label{chilip}
\eps^{-1}\chi_\eps(|x+j(x,z)-y-j(y,z)|)-\eps^{-1}\chi_\eps(|x-y|)\leq D_j|z|\eps^{-1}|x-y|=|z|o_\eps(1),
\end{equation}
where $D_j$ is defined in (J1) and then 
\begin{multline}\label{maxdopodelta2}
u(x+j(x,z))-u(x)-(v(y+j(y,z))-v(y)) \\\leq  \phi((x+j(x,z)+y+j(y,z))/2)-\phi((x+y)/2)+ |z|o_\eps(1).
\end{multline}
Then by \eqref{maxdopodelta2} and the first of (M),   we get
\begin{equation}\label{nonloc0}
\mathcal{T}^{\xi'}[J_x\cap J_y]\leq \mathcal{P}_{\xi}-\mathcal{P}_{\xi'}+\mathcal{K}^\xi+o_\eps(1),
\end{equation}
where $o_{\eps}(1)$ is independent of $\xi'$ and
\begin{equation}\label{kxi}
\mathcal{K}^{\xi}=\int_{\footnotesize{\begin{matrix} J_x\cap J_y, \\|z|\geq \xi\end{matrix}}}u(x+j(x,z))-u(x)-(v(y+j(y,z))-v(y)) d\mu(z),
\end{equation}
\begin{equation}\label{pxidef}
\mathcal{P}_{\xi}=\int_{\footnotesize{\begin{matrix} J_x\cap J_y, \\ |z|\leq \xi\end{matrix}}}\phi((x+j(x,z)+y+j(y,z))/2)
-\phi((x+y)/2)d\mu(z)
\end{equation}
and
\begin{equation}\label{phi0}
\mathcal{P}_{\xi'}=o_{\xi'}(1),
\end{equation}
 by (J1), the first of (M) and the Lipschitz character of $\phi$.
 
Now we consider the term $\mathcal{I}^{\xi'}[J_x/J_y]$, defined in \eqref{jxjy}, observing that the same argument works  similarly for $\mathcal{I}^{\xi'}[J_y/J_x]$.  Take $0<\d_0<1$ enough small (note that $\d_0$ will be defined more precisely at the end of the proof of Lemma \ref{leminpiu}). We split the domain of integration in $\{z\, : \, |z|\geq \d_0\}$  and $\{z\, : \, \xi' \leq|z|\leq \d_0\}$. 
We write
\begin{equation}\label{pallino}
\mathcal{I}^{\xi'}[J_x/J_y]=\mathcal{I}^{\xi'}[B^c_{\d_0}]+\mathcal{I}^{\xi'}[B_{\d_0}],
\end{equation}
where
$$
\mathcal{I}^{\xi'}[B^c_{\d_0}]=\int_{\footnotesize{\begin{matrix} J_x/J_y, \\ |z|\geq \d_0 \end{matrix}}} u(x+j(x,z))-u(x) d\mu(z),
$$
$$
\mathcal{I}^{\xi'}[B_{\d_0}]= \int_{\footnotesize{\begin{matrix}J_x/J_y, \\ \xi'\leq|z|< \d_0 \end{matrix}}} u(x+j(x,z))-u(x) d\mu(z).
$$
By the boundedness of $u$, we have
$$
\mathcal{I}^{\xi'}[B^c_{\d_0}]\leq 
2C||u||_{\infty}\int_{|z|\geq \d_0} 1_{\footnotesize{\begin{matrix}J_x/J_y \end{matrix}}} d\mu(z),
$$
and since
$
|J_x/J_y|\to 0$ as $\eps \to 0,
$
by the first of (M) and  the Dominated Convergence theorem,  we get
\begin{equation}\label{i1d0}
\mathcal{I}^{\xi'}[B^c_{\d_0}]\leq o_{\eps}(1),
\end{equation}
where $o_\eps(1)$ is independent of $\xi'$. 
For $\mathcal{I}^{\xi'}[B_{\d_0}]$ we use again the maximum point inequality \eqref{maxdopod0} with $x'=x+j(x,z), y'=y$ and since $\phi \in C^1$ and by (J1), the first of (M),  we get 
\begin{equation}\label{i2d0}
\mathcal{I}^{\xi'}[B_{\d_0}]\leq C\eps^{-1}\int_{\footnotesize{\begin{matrix}J_x/J_y, \\ \xi' \leq |z|\leq \d_0 \end{matrix}}} \frac{dz}{|z|^{N+\sigma-1}},
\end{equation}
where we remark that $C>0$ is independent of all the parameters.
We couple \eqref{123}, \eqref{nonloc0}, \eqref{phi0}, \eqref{pallino},  \eqref{i1d0} and \eqref{i2d0} with \eqref{123}  and we get
\begin{equation}\label{ineq1}
 \mathcal{I}^{\xi'}[u](x)-\mathcal{I}^{\xi'}[v](y)\leq C\eps^{-1} \mathcal{I}^{\xi'}[J_x/J_y]+C\eps^{-1} \mathcal{I}^{\xi'}[J_y/J_x]+\mathcal{P}_\xi+\mathcal{K}^{\xi}+o_{\eps}(1)+ o_{\xi'}(1),
\end{equation}
where  for all $x,y \in \mathbb{R}^N$,
we denote
\begin{equation}\label{i1eps}
 \mathcal{I}^{\xi'}[J_{x}/J_{y}]:=\int_{\footnotesize{\begin{matrix}J_{x}/J_y,\\ \xi'\leq |z|\leq \d_0 \end{matrix}}} \frac{dz}
{|z|^{N+\sigma-1}}.
\end{equation}
Now we estimate the term in \eqref{i1eps} by Lemma \ref{claimlip}.
Let $r:=r(x_0)$, where $r(x_0)$ is defined in assumption (O) for $\hat{s}=x_0$. Take $rC_j^{-1}/2>\d_0$. Note that,  by \eqref{epsdsmall},  $(x,y)$ satisfy \eqref{ipolem} for  $\hat{s}=x_0$ and $r=r(x_0)$. Then we apply Lemma \ref{claimlip} by taking $\{s_1,s_2\}=\{x,y\},\hat{s}=x_0$ in order to estimate $\mathcal{I}^{\xi'}[J_{x}/J_{y}],\mathcal{I}^{\xi'}[J_{y}/J_{x}]$ defined in \eqref{i1eps} and we get 
\begin{equation}\label{claimapplied}
\mathcal{I}^{\xi'}[J_{x}/J_{y}]\leq C|x-y|, \quad \mathcal{I}^{\xi'}[J_{y}/J_{x}]\leq C|x-y|.
\end{equation} 
Then the claim of the lemma follows by plugging \eqref{claimapplied} into \eqref{ineq1}.
\end{proof}
Note that Lemma \ref{claimlip} is not necessary when dealing with domains with flat boundary.
In the following remark we consider the case when $\Omega$ is the half space and we show how  the estimate of the nonlocal terms can be  carried out more easily  without  Lemma \ref{claimlip}.


\begin{oss}\label{flat}\rm{
Take 
$
\Omega:= \{(x_1, \cdots, x_N=(x',x_N) \in \mathbb{R}^N\, : \, x_N >0\}.
$
For simplicity, we suppose  that 
$
j(x,z)=z \mbox{ if } x+z\in \bar{\Omega}.
$
Note that  (i),(ii) of Lemma \ref{lemdxdy} read
\begin{equation}\label{dxdyhalf}
\bar{x}_N- \bar{y}_N \to 0\quad \mbox{as } \d \to 0.
\end{equation}
Consider the nonlocal terms in \eqref{eq} and restrict ourselves to a subsequence such that $\bar{x}_N \geq \bar{y}_N$ (if $\bar{x}_N \leq \bar{y}_N$  the argument is similar). Then we can write
\begin{eqnarray*}
\mathcal{I}^{\xi'}[u](\bar{x})-\mathcal{I}^{\xi'}[v](\bar{y})\nonumber &=&\int_{\footnotesize{\begin{matrix}-\bar{x}_N \leq z_N < -\bar{y}_N, \\|z|\geq \xi' \end{matrix}}} [u(\bar{x}+z)-u(\bar{x})]d\mu_{\bar{x}}(z)\\  &+&\int_{\footnotesize{\begin{matrix}-\bar{y}_N\leq z_N, \\ |z|\geq \xi'\end{matrix}}} [u(\bar{x}+z)-v(\bar{y}+z)-(u(\bar{x})-v(\bar{y}))]d\mu_{\bar{x}}(z)\\ \nonumber &:=&\mathcal{I}^{\xi'}[J_{\bar{x}}/J_{\bar{y}}]+\mathcal{T}^{\xi'}[J_{\bar{x}}\cap J_{\bar{y}}],
\end{eqnarray*}
where in the last line we used the same notation as in the previous step, see in particular  \eqref{jxjy}, \eqref{jxujy}.
The term $\mathcal{T}^{\xi'}[J_{\bar{x}}\cap J_{\bar{y}}]$ is treated exactly as in the non flat case (see the previous step). On the contrary, note that in this case the estimate of the term $\mathcal{I}^{\xi'}[J_{\bar{x}}/J_{\bar{y}}]$ is easier, since by \eqref{dxdyhalf}
$
|J_{\bar{x}}/J_{\bar{y}}|\to 0$ as $ \d \to 0
$
 and then by the Dominated Convergence Theorem, we have
$
\mathcal{I}^{\xi'}[J_{\bar{x}}/J_{\bar{y}}]\to0 $ as $ \d \to 0. 
$
}\end{oss}
\end{step4}

\begin{step5}-\textit{Sending the other parameters to their limits.}
\upshape
We couple \eqref{eq003} with \eqref{leminpiu} and we get
$$
u(x)-v(y)-B|D\phi((x+y)/2)|\leq C\eps^{-1}|x-y|+\mathcal{P}_\xi+\mathcal{K}^{\xi}+o_{\eps}(1)+o_{\xi'}(1),
$$
where  $C>0$ is independent of the parameters. Then,  we first send $\xi' \to 0$ by the Dominated Convergence Theorem and we get  
\begin{equation}\label{last1}
u(x)-v(y)-B|D\phi((x+y)/2)|\leq C\eps^{-1}|x-y|+\mathcal{P}_\xi+\mathcal{K}^{\xi}+o_{\eps}(1),
\end{equation}
where $C$ is a constant independent of $\xi$. 
Moreover, since $\phi$ is $C^1$, by the first of (M), the Dominated Convergence Theorem and since $x,y \to x_0$ as $\eps \to 0$, we have
$$
\limsup_{\eps \to 0}\mathcal{P}_\xi\leq \mathcal{I}_\xi[\phi](x_0)
$$
and by the boundedness and semicontinuity of $u, v$  and Lemma \ref{lemdxdy} (i), we apply Fatou's lemma for each $\xi>0$ fixed and we get
$$
\limsup_{\eps\to0}\mathcal{K}^{\xi}\leq \mathcal{I}^{\xi}[\omega(\cdot, t_0)](x_0),
$$
and, by the previous estimates, we conclude  by sending  $\eps \to 0$ in \eqref{last1}. $\Box$
\end{step5}
\end{prooflbwb}

\begin{oss}\label{ossmux}\rm{
We give some details of the analysis of the nonlocal terms in step $4$ when the measure $\mu$ depends on $x$. We write \eqref{123} with
$$
\mathcal{I}^{\xi'}[J_x/J_y]=\int_{\footnotesize{\begin{matrix} J_x/J_y, \\  |z|\geq \xi' \end{matrix}}} u(x+j(x,z))-u(x) d\mu_x(z),
$$
$$
\mathcal{I}^{\xi'}[J_y/J_x]= \int_{\footnotesize{\begin{matrix} J_y/J_x, \\ |z|\geq \xi'\end{matrix}}} v(y)-v(y+j(y,z)) d\mu_y(z),
$$
$$
\mathcal{T}^{\xi'}[J_x\cap J_y]=\int_{\footnotesize{\begin{matrix}J_x\cap J_y, \\ |z|\geq \xi'\end{matrix}}} [u(x+j(x,z))-u(x)]d\mu_{x}(z) -[(v(y+j(y,z))-v(y))]d\mu_{y}(z).
$$
For $\mathcal{I}^{\xi'}[J_x/J_y]$ and $\mathcal{I}^{\xi'}[J_y/J_x]$ we proceed as above (Step $4$), noting that the $x$-dependence  plays  no role by the first of (M). 
For the $\mathcal{T}$-term, we write
$$
\mathcal{T}^{\xi'}[J_x\cap J_y]=\mathcal{T}^{\xi'}_1[J_x\cap J_y]+\mathcal{T}^{\xi'}_2[J_x\cap J_y],
$$
where
$$
\mathcal{T}^{\xi'}_1[J_x\cap J_y]=\int_{\footnotesize{\begin{matrix} J_x\cap J_y, \\ |z|\geq \xi'\end{matrix}}} u(x+j(x,z))-u(x) -(v(y+j(y,z))-v(y))d\mu_{y}(z),
$$
$$
\mathcal{T}^{\xi'}_2[J_x\cap J_y]=\int_{\footnotesize{\begin{matrix}J_x\cap J_y, \\ |z|\geq \xi'\end{matrix}}} [(u(x+j(x,z))-u(x))](d\mu_{x}(z)-d\mu_{y}(z)).
$$
For $\mathcal{T}^{\xi'}_1[J_x\cap J_y]$, we proceed as above (in Step $3$, for $\mathcal{T}^{\xi'}[J_x\cap J_y]$ defined in \eqref{jxujy}) and we prove \eqref{nonloc0}. 
 Now consider $\mathcal{T}^{\xi'}_2[J_x\cap J_y]$ and denote
$$
\mathcal{T}^{\xi'}_2[J_x\cap J_y]=\mathcal{T}^{\xi'}_2[B_{\xi}]+\mathcal{T}^{\xi'}_2[B_{\xi}^c],
$$  
where
$$
\mathcal{T}^{\xi'}_2[B_{\xi}]=\int_{\footnotesize{\begin{matrix}J_x\cap J_y, \\ \xi\geq|z|\geq \xi'\end{matrix}}} [(u(x+j(x,z))-u(x))](d\mu_{x}(z)-d\mu_{y}(z)),
$$
$$
\mathcal{T}^{\xi'}_2[B_{\xi}^c]=\int_{\footnotesize{\begin{matrix}J_x\cap J_y, \\ |z|> \xi\end{matrix}}} [(u(x+j(x,z))-u(x))](d\mu_{x}(z)-d\mu_{y}(z)).
$$
For $\mathcal{T}^{\xi'}_2[B_\xi]$ we use the maximum point inequality \eqref{maxdopod0} and we write for $|z|\leq \xi$
\begin{multline}\label{maxsolox}
u(x+j(x,z))-u(x)\leq \eps^{-1}\chi_\eps(|x+j(x,z)-y|)-\eps^{-1}\chi_\eps(|x-y|)\\+\phi((x+j(x,z)+y)/2)-\phi(x+y)/2.
\end{multline}
Then by the Lipschitz continuity of $\chi_\eps$ and $\phi$, (J1), (M) and \eqref{epszero} we get
\begin{equation}\label{I30}
\mathcal{T}^{\xi'}_2[B_{\xi}]\leq C\int_{\scriptsize{\begin{matrix}J_x\cap J_y, \\ \xi\geq|z|\geq \xi' \end{matrix}}}(\eps^{-1}|z|+|z|)(d\mu_{x}(z)-d\mu_{y}(z))\leq o_{\eps}(1),
\end{equation}
where we observe  $o_\eps(1)$ is independent of $\xi'$ and from now on may change from line to line in the following. 
For $\mathcal{T}^{\xi'}_2[B_{\xi}^c]$, by the boundedness of $u$, (M), \eqref{epszero}, we get
\begin{equation}\label{I300}
\mathcal{T}^{\xi'}_2[B_{\xi}^c]\leq 2||u||_{\infty}\int_{\scriptsize{\begin{matrix}J_x\cap J_y, \\|z|> \xi \end{matrix}}}(d\mu_{x}(z)-d\mu_{y}(z))\leq o_\eps(1).
\end{equation}
Then, by \eqref{I30} and \eqref{I300}, we get
\begin{equation}\label{I3}
\mathcal{T}^{\xi'}_2[J_x\cap J_y] \leq o_{\eps}(1),
\end{equation}
where $o_\eps(1)$ is independent of $\xi'$. From now on the proof is the same as  above.}
\end{oss}

\begin{oss}\label{proofa}\rm{
We give the details of the proof of Proposition \ref{lembell} in case $(a)$, when   $x_0\in \Gamma_{\footnotesize{\mbox{in}}}$ is a  strict maximum point of $\omega-\phi=u-v-\phi$ in $\bar{B}_{Cj\xi}(x_0)\cap \bar{\Omega}$, for $\phi \in C^1({\R^N})$ and some $0<\xi<1$.

The strategy of the proof relies on the existence of a  blow-up supersolution exploding on the boundary, which allows us to keep the maximum points away from the boundary. The existence of such a supersolution is stated in the following lemma, whose proof is given in the Appendix.

\begin{lem}\label{lemblowup}
For any $\bar{x} \in \Gamma_{\mbox{\footnotesize{in}}}$, there exists  $r=r(\bar{x})>0$ 
and a positive function  $U_{r} \in C^2(B_{r}(\bar{x})\cap \Omega)$ satisfying for any $\xi$ small enough (with respect to $r$, that is, $\xi<C_j^{-1}\frac{r}{2}$)
\begin{itemize}
\item[(i)] 
$
-b(x,\alpha)\cdot DU_r -\mathcal{I}_\xi[U_r](x)\geq 0 \quad \mbox{ in }  B_{\frac{r}{2}}(\bar{x})\cap\Omega, \quad \forall \alpha \in \mathcal{A};
$
\item[(ii)] 
$
U_r(x)\geq \frac{1}{\omega_{r}(d(x))} \quad \mbox{in }B_{r}(\bar{x}) \cap \Omega,
$
 for some function $\omega_{r}$ which is non-negative, continuous, strictly increasing in a neighbourhood of $0$ and satisfies $\omega_r(0)=0$. 
\end{itemize}
\end{lem}

\begin{proof}[Proof of case $(a)$]
 Let $r=r(x_0)$ be defined in Lemma \ref{lemblowup} for $\bar{x}=x_0$. We localize the argument in a ball of radius $r$ around $x_0$ and we use the existence of the blow-up function $U_r$ defined  in Lemma \ref{lemblowup} for $\bar{x}=x_0$.
 Let  $0<\xi'<\min\{\xi, C_j^{-1}\frac{r}{4}\}$ and $\eps>0$. We double the variable and we consider $(x,y)$ maximum point on $\bar{B}_{C_j \xi'}(x_0)\cap \bar{\Omega} \times \bar{B}_{C_j\xi'}(x_0)\cap \bar{\Omega}$ of the function
$
\Phi(x,y)= u(x)-v(y)-\tilde{\phi}(x,y),
$
where
$$
\tilde{\phi}(x,y)=\phi\left(\frac{x+y}{2}\right)+\frac{|x-y|^2}{\eps^2}+k[U_r(x)+U_r(y)].
$$
Note that, by (ii) of Lemma \ref{lemblowup}, we have that  $(x,y)\in\bar{B}_{C_j\xi'}(x_0)\cap \Omega \times \bar{B}_{C_j\xi'}(x_0)\cap \Omega$; moreover, again by (ii) of Lemma \ref{lemblowup}, we have for $k$ small enough 
\begin{equation}\label{distaxybd}
d(x), d(y)\geq \omega_r^{-1}\left(\frac{k}{2L}\right)=: \bar{\d},
\end{equation}
where
$
L=||u||_{L^\infty(\bar{B}_{C_j}(x_0)\cap \bar{\Omega})}+||v||_{L^\infty(\bar{B}_{C_j}(x_0)\cap \bar{\Omega})}+||\phi||_{L^\infty(\bar{B}_{C_j}(x_0)\cap \bar{\Omega})}+1.
$
Note that the existence of the blow-up function plays its mayor role here to get \eqref{distaxybd}. This estimate tells us, roughly speaking, that the maximum points are away from the boundary.
For fixed $k$, a standard argument shows that
\begin{equation}\label{epszeroa}
\frac{|x-y|^2}{\eps^2} \to 0 \mbox{ as } \eps \to 0.
\end{equation}
By the previous estimate on $x,y$ and extracting subsequences if necessary, we can assume, without loss of generality, that as $\eps, k \to 0$
\begin{equation}\label{convxykeps}
x, y\to x_0, \quad u(x)-v(y)-\tilde{\phi}(x,y) \to u(x_0)-v(x_0)-\phi(x_0).
\end{equation}
Thanks to \eqref{convxykeps}, we can  take $\eps,k$ small enough so that  $x,y \in B_{\frac{r}{4}}(x_0)\cap \Omega$.  We proceed as in Step $2$ in the above proof, we write the viscosity inequalities \eqref{eq} and, using that $\tilde{\phi} \in C^1$, (J1) and the first of (M), we get 
\begin{eqnarray}\label{eqastrana}
u(x)-v(y)\nonumber &\leq& H(y,-D[\tilde{\phi}(x,\cdot)](y))-H(x,D[\tilde{\phi}(\cdot,y)](x)\\ &+&\mathcal{I}^{\xi'}[u](x)-\mathcal{I}^{\xi'}[v](y)+o_{\xi'}(1),
\end{eqnarray}
where $\lim_{\xi'\to 0} o_{\xi'}(1)=0$ for each $k$ fixed and $o_{\xi'}(1)$  is independent of $\d$.
First we analyse the term $\mathcal{I}^{\xi'}[u](x)-\mathcal{I}^{\xi'}[v](y)$. For simplicity of exposition, we conclude the proof in the case the measure $\mu$ in the nonlocal terms has no dependence on $x$, i.e. $\mu_x\equiv \mu$. The result can be easily extended in the case of $x$-dependence analogously as already shown  in Remark \ref{ossmux} for case $(b)$ and $(c)$.
We use the same notation of Proposition \ref{lembell}, see  \eqref{jxjy}, \eqref{jxujy} and we write
\begin{equation}\label{plugginginto}
\mathcal{I}^{\xi'}[u](x)-\mathcal{I}^{\xi'}[v](y)= \mathcal{I}^{\xi'}[J_x/J_y]+\mathcal{I}^{\xi'}[J_y/J_x] +\mathcal{T}^{\xi'}[J_x\cap J_y]. 
\end{equation}
As in the proof of Proposition \ref{lembell}, the term $\mathcal{T}^{\xi'}[J_x\cap J_y]$ can be estimated as follows  
\begin{eqnarray*}
\mathcal{T}^{\xi'}[J_x\cap J_y]\nonumber &\leq& k\mathcal{I}_{\xi}[U_r](x)+k\mathcal{I}_{\xi}[U_r](y)-k\mathcal{I}_{\xi'}[U_r](x)-k\mathcal{I}_{\xi'}[U_r](y)\\ &+&\mathcal{P}_{\xi}+\mathcal{K}^\xi+o_\eps(1)+o_{\xi'}(1),
\end{eqnarray*}
where we use the notation \eqref{kxi}, \eqref{pxidef} of Proposition \ref{lembell}. Note that  $o_{\eps}(1)$ is independent of $\xi'$.
Since $U_r$ is Lipschitz, by (J1) and the first of (M), we  have
$
\mathcal{I}_{\xi'}[U_r](x),  \mathcal{I}_{\xi'}[U_r](y)\leq o_{\xi'}(1),
$
and then 
\begin{equation}\label{unionek}
\mathcal{T}^{\xi'}[J_x\cap J_y]\leq k\mathcal{I}_{\xi}[U_r](x)+k\mathcal{I}_{\xi}[U_r](y)+\mathcal{P}_{\xi}+\mathcal{K}^\xi+o_{\xi'}(1)+o_\eps(1),
\end{equation}
where $o_\eps(1)$ is independent of $\xi'$.
Now we estimate the terms  $\mathcal{I}^{\xi'}[J_x/J_y]$ and $\mathcal{I}^{\xi'}[J_y/J_x]$.  Thanks to  \eqref{distaxybd}, in this case  the estimate  is easier than in the  cases (b) and (c) treated above.
Take for example $\mathcal{I}^{\xi'}[J_x/J_y]$ (the argument being analogous for $\mathcal{I}^{\xi'}[J_y/J_x]$) and note that by \eqref{distaxybd} the integral is independent of $\xi'$ as soon as $\xi'<\bar{\d}$ where $\bar{\d}$ is defined in \eqref{distaxybd}.  Then by the boundedness of $u$, we have
$$
\mathcal{I}^{\xi'}[J_x/J_y]\leq 
2||u||_{\infty}\int_{|z|\geq \bar{\d}} 1_{\footnotesize{\begin{matrix}J_x/J_y \end{matrix}}} d\mu(z)
$$
and since
$
|J_x/J_y|\to 0 \mbox{ as } \eps \to 0,
$
by the first of (M), the Dominated Convergence theorem, we get
\begin{equation}\label{estafacile}
\mathcal{I}^{\xi'}[J_x/J_y]\leq o_\eps(1),
\end{equation}
where $o_\eps(1)$ is independent of $\xi'$. 
Then plugging \eqref{estafacile} and \eqref{unionek} into \eqref{plugginginto} and then coupling it with \eqref{eqastrana},  we get 
\begin{eqnarray}\label{eqa}
u(x)-v(y)\nonumber &\leq& H(y,-D[\tilde{\phi}(x,\cdot)](y))-H(x,D[\tilde{\phi}(\cdot,y)](x)\\ &+&k\mathcal{I}_{\xi}[U_r](x)+k\mathcal{I}_{\xi}[U_r](y)+\mathcal{P}_{\xi}+\mathcal{K}^\xi+ o_{\eps}(1)+o_{\xi'}(1),
\end{eqnarray}
where $o_\eps(1)$ is independent of $\xi'$.
Now,  by (i) of Lemma \ref{lemblowup}, we  estimate the  integral terms of the left hand side of \eqref{eqa} together with the first order terms involving $U_r$  in  $H(y,-D[\tilde{\phi}(x,\cdot)](y))-H(x,D[\tilde{\phi}(\cdot,y)](x)$ and we get
\begin{multline}\label{esthama}
H(y,-D[\tilde{\phi}(x,\cdot)](y))-H(x,D[\tilde{\phi}(\cdot,y)](x)+k\mathcal{I}_{\xi}[U_r](x)+k\mathcal{I}_{\xi}[U_r](y)\\ \leq B|D\phi((x+y)/2)|+o_{\eps}(1).
\end{multline}
Then, plugging \eqref{esthama} into \eqref{eqa}, we get
\begin{equation}\label{eq003a}
u(x)-v(y)\leq B|D\phi((x+y)/2)|+\mathcal{P}_{\xi}+\mathcal{K}^\xi+o_{\eps}(1)+o_{\xi'}(1),
\end{equation}
where $o_\eps(1)$ is independent of $\xi'$. 
The rest of the proof is the same as in the previous cases, by sending first $\xi'\to 0$, then $\eps\to 0$ and finally $k \to 0$. For the details we refer to the end of the proof of Proposition \ref{lembell}.
\end{proof}
}
\end{oss}

Now we prove Theorem \ref{thm} for $H$ of Bellman type.

\begin{proof1}\rm{
By contradiction, we suppose that 
$
M=\sup_\Omega\{u-v\}>0.
$
Denote
$
\omega(x)=u(x)-v(x) 
$
and  for $\nu>0$, consider
$
\Phi(x)=\omega(x)-\psi(R^{-1}|x|)+\nu d(x),
$
where $\psi$ is a smooth function such that
\begin{equation}\label{psir0}
\psi(s)= \left\{
 \begin{array}{lll}
 0 \, \, &\mbox{for} \, \, 0 \leq s <\frac{1}{2}, \\
                 \mbox{increasing} \, \, &\mbox{for} \, \, \frac{1}{2}\leq s <1,\\
||u||_{\infty}+||v||_{\infty} +1 \, \, &\mbox{for} \, \, s \geq 1.
                 \end{array}
\right.\,
\end{equation}
and  $d$ is the signed distance from the boundary (see Remark \ref{distance}). Note that $\sup \Phi \to M$ as $R \to \infty$ and $\nu \to 0$. Since $\Phi \leq -1/2$ for $|x|$ large and $\nu$  small enough and $M>0$, the function $\Phi$ achieves its positive maximum $\sup \Phi>\frac{M}{2}$ at a point $x$ for $R$ big  and $\nu$ small enough.  We give the details in the case where all maximum points $x$  are located on the boundary.
We have
\begin{equation}\label{omegaM}
\omega(x)=M+o_{R,\nu}(1),
\end{equation}
where with $o_{R,\nu}(1)$ we mean that $o_{R,\nu}(1)\to 0$ if $ R\to \infty, \nu\to0$. 

We use $\phi(\cdot):=\psi(R^{-1}|\cdot|)-\nu d(\cdot)$  as a test function at $x$. Note that, if $x \in \partial \Omega$ and for $\nu > R^{-1}||\psi'||_{L^{\infty}}$, we have
$
\frac{\partial \phi}{\partial n}\geq -R^{-1}||\psi'||_{L^{\infty}}+\nu>0.
$
Then, by Proposition \ref{lembell}, we get
$$
\omega(x)-\mathcal{I}[\phi](x)- B(|D\phi(x)|)\leq 0 \, \, \mbox{ in } \, \,  \Omega.
$$
By Lemma \ref{primophi} (see the  Appendix), 
$
\mathcal{I}[\phi](\cdot), |D\phi(\cdot)| \leq o_{\nu,R}(1)
$
and by \eqref{omegaM}, we get
$$
M+o_{\nu,R}(1)\leq  o_{\nu,R}(1),
$$
and by letting  $R \to \infty$, $\nu \to0$, we get a contradiction since $M>0$.
 }
 \hspace{\stretch{1}} $\Box$ 
\end{proof1}

\subsection{Coercive Hamiltonians}
We proceed analogously as for Hamiltonians of Bellman type and we prove Proposition \ref{lembellc}.
Once proved Proposition \ref{lembellc}, the proof of Theorem \ref{thm} for $H$ coercive follows by standard arguments as already showed for Hamiltonians of Bellman type. We sketch first the proof of Theorem \ref{thm} and then we prove Proposition \ref{lembellc}.


\begin{proof2}\rm{
We just observe that we proceed again by contradiction, supposing that 
$
M=\sup_\Omega\{u-v\}>0.
$
We fix $0<\mu<1$ and define  for $ x \in \Omega$
$
\omega_\mu(x)=\mu u(x)-v(x).
$
We proceed  as in the Bellman case  and  we use Proposition \ref{lembellc} to get
$$
M+o_{\nu,R,\mu}(1)- o_{\nu,R}(1)\leq CA(1-\mu).
$$
Then, by letting  $R \to \infty$, $\nu \to0$ and finally $\mu \to 1$, we get a contradiction since $M>0$ and we conclude the proof.
}\hspace{\stretch{1}} $\Box$ 
\end{proof2}

\begin{prop}\label{lembellc}
Let $\mathcal{I}$ as in \eqref{idol}  and assume  $(M)$, $(J0), (J1)$. Let $H$ be a  coercive Hamiltonian  satisfying (H1), (Ha)  or (H1), (Hb), (Hc) and let $u,v$ be respectively bounded  sub and supersolutions to \eqref{solutions0}.  Let $\mu \in (0,1)$ if $H$ is superlinearly coercive, $\mu=1$ is $H$ is sublinearly coercive.
Then the function
$
\omega(x):=\mu u(x)-v(x)
$
satisfies, in the viscosity sense, the equation
\begin{equation}\label{res}
\left\{
 \begin{array}{ll}
 \omega-\mathcal{I}[\omega](x)- C_{m,\mu}|D\omega|^m\leq A(1-\mu) &\mbox{in} \, \,  \Omega, \\
                 \frac{\partial \omega}{\partial n}= 0 \, \, &\mbox{on} \, \, \partial \Omega,\\
                 \end{array}
\right.\,
\end{equation}
where $A, C_{m,\mu}$ are  positive constants  which depend on the data. Precisely, if $\mu=1$, $C_{m,\mu}=\tilde{C}$ where $\tilde{C}$ is defined  in (Hb)  and, if $\mu\in(0,1)$, $C_{m,\mu}=\bar{C}^{1-m}C^m2^{m-1}m^{-m}(1-\mu)^{1-m}$ 
\end{prop}

\begin{prooflbcwb}
We give the details when $H$ has  superlinear form i.e. when $m >1>\sigma$, since the proof in the sublinear case is similar with easier computations. Since the proof is similar to that of Proposition \ref{lembell}, we  focus  only on the main differences.

We start by noting that if $u$ is a subsolution of \eqref{solutions0}, then $
\bar{u}=\mu u
$ is a viscosity subsolution to
\begin{equation}\label{baru}
\bar{u}-\mathcal{I}[\bar{u}](x)+\mu H(x,\mu^{-1}D \bar{u})\leq 0, \quad \mbox{in } \Omega.
\end{equation}
Let $x_0\in \bar{\Omega}$ and $\phi$ a smooth function such that $\omega-\phi$ has a strict maximum point in $\bar{B}_{C_j\xi}(x_0)\cap \bar{\Omega}$ at $x_0$ for some  $0<\xi<1$. 
We suppose that $x_0 \in \partial \Omega$, since the other case being similar and even simpler. 
\begin{step1}-\textit{Localising on equidistant points  (i.e. $d(x)=d(y)$).}
\upshape
We  double the variable and consider the function
\begin{equation}\label{phiprimad}
\Phi(x,y):=\bar{u}(x)-v(y)-\tilde{\phi}(x,y),
\end{equation}
where $\tilde{\phi}$ is as in \eqref{phid} with $K=2$.  Let $(\bar x, \bar y)$  be a maximum point  of $\Phi$  over the set
$
A:=\bar{B}_{2C_j\xi'}(x_0)\cap \bar{\Omega}\times \bar{B}_{2C_j\xi'}(x_0)\cap \bar{\Omega},
$
for $0<\xi'<\frac{\xi}{2}$. Let $\Psi=\bar{u}(x)-v(y)-\psi(x,y)$, where $\psi$ is defined as in \eqref{testpsi0} with $K=2$ and let $(x,y)$ be a point of maximum of $\Psi$ over $A$. By classical argument, we have as $\eps \to 0$
\begin{equation}\label{eqinpiu}
x,y \to x_0, \quad \eps^{-1}\chi_\eps(|x-y|)\to 0, \quad \eps^{-1}|d(x)-d(y)|\to 0.
\end{equation}
We have the following lemma.
\begin{lem}\label{lemdxdycoer}
Under the above notation, we have
\begin{itemize}
\item[(i)] $\bar{x} \to x, \, \,\bar{y} \to y, \,\, u(\bar{x})\to u(x), \,\, v(\bar{y}) \to v(y)\,\, \mbox{ as } \d \to 0,$
\item[(ii)] $d(x)=d(y).$
\end{itemize}
\end{lem}
\begin{proof}
The proof is very similar to that of Lemma \ref{lemdxdy}. We give a sketch of the proof of (ii). We suppose that
$
\frac{\partial \phi}{\partial n}(x_0)>0,
$
then for $\eps$ small enough 
$
\frac{\partial \psi}{\partial n}((x+y)/2)>\frac{1}{2}\frac{\partial \phi}{\partial n}(x_0)>0.
$
We assume that
$
d(x) > d(y).
$
Then  for $0<\mu<1$, we have
\begin{equation}\label{muu}
\bar{u}(x)-\mathcal{I}^{\xi'}[\bar{u}](x)-\mathcal{I}_{\xi'}[\psi(\cdot,y)](x) +\mu H\left(x,\mu^{-1}D[\psi(\cdot,y)](x)\right)\leq 0.
\end{equation}
Note that
\begin{equation}\label{dphicoerc}
|D[\psi(\cdot,y)](x)|\geq\eps^{-1}-C,
\end{equation}
where  $C>0$ is a constant independent of $\eps$. For the integral terms in \eqref{muu} we proceed as in Proposition \ref{lembell} by using Lemma \ref{estnonloclem}. For the Hamiltonian terms we use assumption (H1) together with \eqref{dphicoerc} and we get
$$
\eps^{-m}\left(-\eps^{m-\sigma}+ c_0\mu^{1-m}\right) \leq C,
$$
which is a contradiction for $\eps $ small, since  $\sigma < m$.
\end{proof}
\end{step1}

\begin{step2}-\textit{Writing the viscosity inequalities.}
\upshape
 We write  the viscosity inequalities for $u$ and $v$  
\begin{eqnarray}\label{ineq1c}
\bar{u}(\bar{x})-v(\bar{y})\leq\mu H(\bar{y},\mu^{-1}D_x \tilde{\phi}(\cdot,\bar{y})(\bar{x}))\nonumber &-&H(\bar{x},-D_y \tilde{\phi}(\bar{x},\cdot)(\bar{y}))\\ &+&\mathcal{I}^{\xi'}[\bar{u}](\bar{x})-\mathcal{I}^{\xi'}[v](\bar{y})+o_{\xi'}(1),
\end{eqnarray} 
where we estimated the $\mathcal{I}_{\xi'}[\phi]$ terms by $o_{\xi'}(1)$ (independent of $\d$)  as in \eqref{oxi}.
Denote
\begin{equation}\label{asterisco}
\mathcal{H}=\mu H(\bar{y},\mu^{-1}D[\tilde{\phi}(\cdot,\bar{y})](\bar{x}))-H(\bar{x}, -D[\tilde{\phi}(\bar{x},\cdot)](\bar{y})).
\end{equation}
We recall that
\begin{equation}\label{phi1}
D[\tilde{\phi}(\cdot,\bar{y})](\bar{x})=\eps^{-1}[\chi'_\eps(|\bar{x}-\bar{y}|)\hat{p}-2\chi'_\d(|d(\bar{x})-d(\bar{y})|)\tilde{p}n(\bar{x})] +D\phi((x+y)/2)/2,
\end{equation}
\begin{equation}\label{phi2}
D[-\tilde{\phi}(\bar{x},\cdot)](\bar{y})=\eps^{-1}[\chi'_\eps(|\bar{x}-\bar{y}|)\hat{p}-2\chi'_\d(|d(\bar{x})-d(\bar{y})|)\tilde{p}n(\bar{y})] -D\phi((x+y)/2)/2,
\end{equation}
where 
\begin{equation}\label{hatpdelta}
\hat{p}=\frac{\bar{x}-\bar{y}}{|\bar{x}-\bar{y}|}, \quad \tilde{p}=\frac{d(\bar{x})-d(\bar{y})}{|d(\bar{x})-d(\bar{y})|}.
\end{equation}
By the smoothness of the distance function 
\begin{equation}
\left|D[\tilde{\phi}(\cdot,\bar{y})](\bar{x})-D[-\tilde{\phi}(\bar{x},\cdot)](\bar{y})\right|\leq 4\eps^{-1}|\bar{x}-\bar{y}| +|D\phi((\bar{x}+\bar{y})/2)|,
\end{equation}
and
$$
|D[\tilde{\phi}(\cdot,\bar{y})](\bar{x})|,|D[-\tilde{\phi}(\bar{x},\cdot)](\bar{y})|\leq q, \quad
q=3\eps^{-1}+2^{-1}||D\phi||_{L^{\infty}(B_{r}(x_0))}.
$$
Thanks to (Hb) and (Hc), we get for $\eps$ small
\begin{multline*}
\mathcal{H}\geq (1-\mu)(m-1)\bar{C}q^m -A(1-\mu)-\omega_1(|\bar{x}-\bar{y}|)(1+q^m)\\-C|D\phi((\bar{x}+\bar{y})/2)|q^{m-1}-4C|\bar{x}-\bar{y}|q^{m}.
\end{multline*}
Note that, by \eqref{eqinpiu} and (iii) of Lemma \ref{lemdxdycoer}, we can take $\epsilon=\epsilon(\mu), \delta=\delta(\mu)$ small enough so that
$$
(1-\mu)(m-1)\bar{C}-\omega_1(|\bar{x}-\bar{y}|)-4C\eps^{-1}|\bar{x}-\bar{y}|>0
$$
and we can write
\begin{eqnarray*}
\mathcal{H} &\geq& (1-\mu)(m-1)\bar{C}q^m/2-C|D\phi((\bar{x}+\bar{y})/2)|q^{m-1}-A(1-\mu)-o_{\d,\eps}(1) \\&\geq&\inf_{q\geq0} \{(1-\mu)(m-1)\bar{C}q^m/2-C(|D\phi((\bar{x}+\bar{y})/2)|)q^{m-1}\}-A(1-\mu)-o_{\d,\eps}(1),
\end{eqnarray*}
where $o_{\d, \eps}(1)$ means that $\lim_{\d \to 0} o_{\d,\eps}(1)=o_{\eps}(1)$. Note that the infimum in the previous expression is attained and therefore
\begin{equation}\label{mathcalh}
\mathcal{H} \geq -C_{m,\mu}|D\phi((\bar{x}+\bar{y})/2)|^m-A(1-\mu)-o_{\d,\eps}(1),
\end{equation}
where $C_{m,\mu}=\bar{C}^{1-m}C^m2^{m-1}m^{-m}(1-\mu)^{1-m}.$ 
Then we couple \eqref{mathcalh}, \eqref{asterisco} and  \eqref{ineq1c}, we let $\d \to 0$ and we get
\begin{eqnarray}\label{eqfinc}
\bar{u}(x)-v(y)\nonumber &-&C_m(1-\mu)^{1-m}|D\phi((x+y)/2)|^m-A(1-\mu)-o_{\eps}(1)\\ &\leq&\mathcal{I}^{\xi'}[\bar{u}](x)-\mathcal{I}^{\xi'}[v](y)+o_{\xi'}(1).
\end{eqnarray}
\end{step2}
\begin{step3}{Estimate of the nonlocal terms.}
\upshape
In order to estimate the nonlocal terms we use the following  lemma. We omit the proof since it is exactly the same as that of Lemma \ref{lemestnonlocb}. We remark that in the proof we deeply rely on the assumption $\sigma \in (0,1)$.
\begin{lem}\label{lemestnonloc}
Under the above notation, we have
\begin{equation}\label{claimlemestnonloc}
\mathcal{I}^{\xi'}[\bar u](x)-\mathcal{I}^{\xi'}[v](y)\leq  C\eps^{-1}|x-y|+\mathcal{P}_\xi+\mathcal{K}^{\xi}+o_{\eps}(1)+o_{\xi'}(1),\end{equation}
where $C>0$ is independent of all the parameters.
\end{lem}
Then the rest of the proof easily follows as in Proposition \ref{lembell}, step $5$.  $\Box$
\end{step3}
\end{prooflbcwb}

\section{Applications to evolutive problems: existence, uniqueness and  asymptotic behavior}\label{appl}
In this section we present some applications  of our results proved in the previous sections for the stationary case to the evolutive setting and we consider the associated Cauchy problem 
\begin{equation}\label{solutions0t}
\left\{
 \begin{array}{ll}
 \partial_t u-\mathcal{I}[u(\cdot, t)](x)+H(x,t, u, Du)=0 \, \, &\mbox{in} \, \,  \Omega \times (0,+\infty),\\
                 \frac{\partial u}{\partial n}= 0 \, \, &\mbox{on} \, \, \partial \Omega \times (0,\infty),\\
                  u(x,0)=u_0(x) \, \, &\mbox{in} \, \, \bar{\Omega},
                 \end{array}
\right.\,
\end{equation}
where  $\Omega\subset \mathbb{R}^N$  satisfies assumption (O), $u_0\in C(\Omega)$, $\mathcal{I}[u]$ is an integro-partial differential operator of censored type and of order strictly less than $1$, defined as  
\begin{equation}\label{idot}
\mathcal{I}[u(\cdot, t)](x)=\lim_{\d \to 0^{+}} \int_{\footnotesize{\begin{matrix}|z|>\d, \\x+j(x,z) \in \bar{\Omega} \end{matrix}}} \, [u(x+j(x,z), t)-u(x,t)]d\mu_x(z),
\end{equation}
where  $\mu_x$ is a singular non-negative Radon measure satisfying (M), and  $j(x,z)$  is a jump function satisfying (J0), (J1) (see Section \ref{sec:ass}) .  The main example are measures $\mu_x$ with density  $\frac{d\mu_x}{dz}=g(x,z)|z|^{-(N+\sigma)}$ with $\sigma <1$  and $g$ a non-negative bounded function Lipschitz  in $x$ uniformly with respect to $z$. 
Note that the operator \eqref{idot} is  the natural extension to the evolutive case  of the nonlocal operator considered in the stationary case and defined in \eqref{idol}.
Moreover
$H\, : \, \bar{\Omega} \times [0,+\infty)\times \R\times \mathbb{R}^N \mapsto \mathbb{R}$ is a continuous function whose growth in the gradient  makes it the leading order term in the  equation, and can be coercive or of Bellman type. We refer to the following   section for the precise assumptions on the Hamiltonian. 
The well-posedness of problems as \eqref{solutions0t} follows from analogous arguments used for the stationary problem with some standard adaptations. We refer  to Theorem \ref{thmcontrolt} and Theorem \ref{ex} for further details and proofs of uniqueness and existence.

We study two different kind of asymptotic behaviour of the solutions of \eqref{solutions0t}.  First we consider a Hamiltonian  either of Bellman type either coercive  under assumptions ensuring uniqueness of the solution of the associated stationary problem. We prove, by classical methods based on the weak-relaxed semilimits, the  convergence as $t \to + \infty$ of the solution of \eqref{solutions0t} to the unique solution of the associated  stationary problem.
 On the other hand, when the associated stationary problem has not unique solution, we consider a Hamiltonian with superfractional coercive growth and  we study the so-called \textit{ergodic large time behaviour}, proving that the solution of \eqref{solutions0t} approaches  a solution of the so-called \textit{ergodic problem} as $t \to + \infty$. 
 We follow the methods of \cite{toppreg}, which  rely on the  H\"older regularity in $\Omega$ of the subsolutions of the associated ergodic problem and on a Strong Maximum Principle. We recall also  \cite{BCCI2}, where analogous methods were applied to uniformly elliptic integro-differential equations. We refer to subsection \ref{secondconvergence} for more details.

\subsection{Assumptions}
We are going to consider the finite time horizon problem associated to \eqref{solutions0t}
\begin{equation}\label{solutions0TT}
\left\{
 \begin{array}{ll}
 \partial_t u-\mathcal{I}[u(\cdot, t)](x)+H(x,t, u, Du)=0 \, \, &\mbox{in} \, \,\Omega \times (0,T],  \\
                 \frac{\partial u}{\partial n}= 0 \, \, &\mbox{on} \, \, \partial \Omega \times (0,T),\\
                  u(x,0)=u_0(t) \, \, &\mbox{in} \, \, \bar{\Omega},
                 \end{array}
\right.\,
\end{equation}
The definition of viscosity solutions to \eqref{solutions0TT} (and then to  \eqref{solutions0t}) is the natural extension of Definition \ref{defsol} to the corresponding Cauchy problem.

We  assume the following condition. 
\begin{enumerate}
\item[(H')] For all $R>0$, there exists $\gamma_R\geq -K$ such that
for all $x \in \bar{\Omega},  u, v \in \mathbb{R},|u|,|v|\leq R, 0\leq  t \leq R$ and $p\in \mathbb{R}^n$, we have
$$
H(x,t, u,p)-H(x,t, v,p)\geq \gamma_R(u-v).
$$
\end{enumerate}


\begin{oss}\label{coerciveprop}\rm{
Note that in the comparison principle (Theorem \ref{thmcontrolt}), assumption (H') is assumed mainly for simplicity of exposition and can be relaxed assuming only $\gamma_R \geq 1$.
  Indeed if $\gamma_R<1$ we perform the change $\tilde{u}=ue^{-(\gamma_R-1)t}$ (analogously for the supersolution) and prove Theorem \ref{thmcontrolt} for $\tilde{u}$ and $\tilde{v}$.}
\end{oss}

As it is classical in viscosity solutions theory, the comparison principle allows the application of Perron's method to conclude the existence. To this end, we introduce the following assumption, which will allows us to build sub and supersolutions:
\begin{enumerate}
\item[(E')] For all $T>0, R>0$  there exists a constant $H_R>0$ such that 
$$
||H(x,t, r, p)||_\infty \leq H_R \, \, \forall x \in \Omega, t \in [0,T], r, p \in \mathbb{R}, |r|, |p| \leq R.
$$
\end{enumerate}
We consider Hamiltonian either of Bellman type either coercive.

We say that the Hamiltonian $H$ is of Bellman type
if for $t \in [0,+\infty), x \in \bar{\Omega}, p \in \mathbb{R}^N, H(x,t, r,p)$ can be written as
\begin{equation}\label{bellman1}
H(x,t,r,p)=\sup_{\alpha \in \mathcal{A}} \{\lambda(x,t,\alpha) r-b(x,t, \alpha)\cdot p -l(x,t, \alpha)\},
\end{equation}
where $b, \lambda\, : \, \bar{\Omega} \times [0,+\infty)\times \mathcal{A} \to \mathbb{R}^N$ and $l\, : \, \bar{\Omega}\times[0,+\infty) \times \mathcal{A}\to \mathbb{R},$
are continuous and  bounded functions and satisfy the following properties.
\begin{enumerate}
\item[(C')] \label{lcont} \textit{Uniform  continuity  of $l$ and $\lambda$}:\\
There exist modulus of continuity  $\omega_l, \omega_\lambda$ such that such that $\forall \alpha \in \mathcal{A}, \forall x, y \in \bar{\Omega}, t,s\in [0,+\infty)$
$$
|l(x,t,\alpha)-l(y,s,\alpha)|\leq \omega_l(|x-y|+|t-s|);
$$
$$
|\lambda(x,t,\alpha)-\lambda(y,s,\alpha)|\leq \omega_\lambda(|x-y|+|t-s|);
$$
\item[(L')] \label{blip} \textit{Uniform Lipschitz continuity of the drift $b$}: There exists $C>0$ such that $\forall \alpha \in \mathcal{A} \, \forall (x,s), (y,t) \in \bar{\Omega}\times [0+\infty)$\\
$$
|b(x,s, \alpha)-b(y,t,\alpha)| \leq C(|x-y|+|s-t|).
$$
\end{enumerate}

\begin{oss}\rm{
Note that assumption (L') may seem unusual since it requires also the uniform Lipschitz continuity of the drift $b$ in the time variable. This is due to the fact that in the proof of the comparison principle (Theorem \ref{thmcontrolt}) we need to double also the time variable in the test function, which as a consequence will have the same dependence on $x$ and $t$, which we estimate by assumption (L'). We refer to Remark \ref{osscompt}. 
}
\end{oss}

We assume that the components of the boundary are of three different types whose definition is analogous to that of $\Gamma_{\mbox{\footnotesize{in}}}, \Gamma_{\mbox{\footnotesize{out}}}, \Gamma$ defined in \eqref{gammain}, \eqref{gammaout}, \eqref{gamma} in the stationary case, with the only difference that now they are part of the parabolic boundary. For simplicity of exposition, we do not repeat the definition. For the sake of simplicity and for the rest of the paper, we adopt the following notation.
\begin{itemize}
\item[(B')]
Assumption (B) where $\Gamma_{\mbox{\footnotesize{in}}}, \Gamma_{\mbox{\footnotesize{out}}}, \Gamma$ are considered as  parts of the parabolic boundary.
 \end{itemize}

In the case of coercive Hamiltonians, we restrict the time dependence of $H$ by the assumption:\\

\begin{enumerate}
\item[(H0')] There exists  $H_0 \, : \, \bar{\Omega}\times \mathbb{R} \times \mathbb{R}^N \to \mathbb{R}$ continuous  and $f\, : \, \bar{\Omega} \times [0,\infty) \to \mathbb{R}$ uniformly continuous  and  bounded such that for all $x \in \bar{\Omega}, r \in \mathbb{R}, p \in \mathbb{R}^N$
$$
H(x,t,r,p)=H_0(x,r,p)-f(x,t).
$$
\end{enumerate}

We assume that the Hamiltonian satisfies in $x$ uniformly in $t$ and uniformly on compact sets in $r$ the assumptions made in the stationary case, namely, (H1), (Ha) (sublinear coercivity) or (H1), (Hb), (Hc) (superlinear coercivity).    For the sake of brevity and simplicity of exposition, we omit to repeat them and for the rest of the paper we adopt the following notation.
\begin{itemize}
\item[(He')] The Hamiltonian satisfies (H1), (Ha) in $x$ uniformly in $t$ and uniformly on compact sets in $r$.  In this case we say  the Hamiltonian satisfies (H1'), (Ha').
\end{itemize}
\begin{itemize}
\item[(He'')] The Hamiltonian satisfies (H1), (Hb), (Hc) in $x$ uniformly in $t$ and uniformly on compact sets in $r$. In this case we say the Hamiltonian satisfies (H1''), (Hb''), (Hc'').
\end{itemize}

\subsection{Existence and uniqueness}\label{subext}
In the following theorem we state the comparison principle for the problem \eqref{solutions0t}.

\begin{thm}\label{thmcontrolt}
Let $\Omega$ be an open subset of $\mathbb{R}^N$ satisfying $(O)$, $u_0\in C(\bar{\Omega})$. Assume (M), (J0), (J1). Let $H$ be  a Hamiltonian  either of Bellman type as in \eqref{bellman1} satisfying (C'), (L'), (B') or a coercive Hamiltonian  satisfying (H0') and either (He') or (He''). Assume also (H').   Let $u, v \in L^{\infty}(\bar{\Omega}\times [0,T])$ for all $T>0$ be respectively a  bounded usc sub  and bounded lsc supersolution of \eqref{solutions0t}. 
Then 
$$
u \leq v \quad \mbox{ in } \bar{\Omega} \times [0,+\infty).
$$
\end{thm}
The proof is based on the following proposition, which is the analogous for the evolutive case to Propositions \ref{lembell} and  \ref{lembellc}. The proof is also similar with some adaptation to the evolutive setting.
Note that  Proposition \ref{eqmut} is used also later in the proof of Proposition \ref{strongmaxpr}.
\begin{prop}\label{eqmut}
Let $\mathcal{I}$ as in \eqref{idot}  and assume (M),(J0), (J1). Let $H$ be an Hamiltonian either in Bellman form as in \eqref{bellman1} satisfying (C'), (L'), (B') or a coercive Hamiltonian  satisfying (H0') and either (He') or (He''). Assume also (H'). Let $u,v \in L^{\infty}(\bar{\Omega}\times [0,T])$ for any $T>0$ be respectively  usc sub and lsc supersolutions to 
$$
\left\{
 \begin{array}{ll}
 \partial_t w-\mathcal{I}[w(\cdot, t)](x)+H(x,t, w, Dw)=0 \, \, &\mbox{in} \, \,\Omega \times (0,+\infty)  \\
                 \frac{\partial w}{\partial n}= 0 \, \, &\mbox{on} \, \, \partial \Omega \times (0,+\infty),\\
                 \end{array}
\right.\,
$$
Take $\mu=1$ if $H$ has Bellman form and if $H$ is sublinearly coercive, $\mu \in (0,1)$ if $H$ is superlinearly coercive. 
Then, the function
$$
\omega(x,t):=\mu u(x,t)-v(x,t)
$$
satisfies, in the viscosity sense, the equation
$$
\left\{
 \begin{array}{ll}
 \partial_t \omega+\gamma_R \omega -\mathcal{I}[\omega(\cdot, t)](x)- C_{m,\mu,R}|D\omega|^m\leq A_R(1-\mu) &\mbox{in} \, \,  \Omega \times (0,+\infty)\\
                 \frac{\partial \omega}{\partial n}= 0 \, \, &\mbox{on} \, \, \partial \Omega \times (0,+\infty),\\
                 \end{array}
\right.\,
$$
where, for $R=||\mu u||_\infty + ||v||_\infty$, $\gamma_R$ is the constant of (H'),  
$A_R>0$  is defined in $(He'')$-$(Hb'')$ and $C_{m,\mu,R}>0$ is such that $C_{m,\mu,R}=:C_{\mu,m}=(1-\mu)^{1-m}$ if $\mu \in (0,1)$ and $C_{m,\mu,R}=\tilde{C}_R$ if $\mu=1$ and $H$ is sublinearly coercive, where $\tilde{C}_R$ is defined in $(He')$-$(Ha')$.
\end{prop}
\begin{oss}\label{osscompt}\rm{
 We omit the proof of Theorem \ref{thmcontrolt}, since it follows closely the arguments presented in Proposition \ref{lembell} and Proposition \ref{lembellc} for the stationary case adapted by standard arguments to the evolution setting. We only remark that, since  the strategy of the proof of Proposition \ref{lembell} relies on an asymmetric use of the viscosity inequalities satisfied by the sub- and supersolution (see in particular in Step $1$),  the standard approach by the Ishi's Lemma for evolutive equation is not applicable. Then, we need to double also the time variable in the test function, which as a consequence will have the same dependence on $x$ and $t$. 
Because of this, in particular in order to estimate the Hamiltonian terms in the viscosity inequalities satisfied by $u$ and $v$ (in the Bellman case), we need the uniform continuity also with respect to time as stated in assumption (L').
}
\end{oss}


For both the coercive and Bellman case, the application of Perron's method on a sequence of finite-time horizon problems with the form \eqref{solutions0TT}  with $T \to +\infty$ and the strong comparison principle allows us to get the existence of a solution which is defined for all time.
Note that, in order to apply Perron's method, we ask the initial datum to be bounded, i.e. $u_0\in C_b(\Omega)$. 

Moreover, in order to have the uniform boundedness in $T$ of the solutions of \eqref{solutions0t}, we suppose the following assumption.
\begin{enumerate}
\item[(H'')]  
There exists $\gamma_0>0$ such that $\gamma_R\geq \gamma_0 $ for all $R>0$ where $\gamma_R$ is defined in (H').
\end{enumerate}

\begin{thm}\label{ex}
Let $\Omega$ be an open subset of $\mathbb{R}^N$ satisfying $(O), u_0\in C_b(\mathbb{R}^N)$. Assume (M), (J0), (J1).  Let $H$ be  a Hamiltonian  either of Bellman type as in \eqref{bellman1} satisfying (C'), (L'), (B') or a coercive Hamiltonian  satisfying  (H0') and either (He') or (He''). Assume also (H'), (E').
 Then, there exists a unique bounded  viscosity solution to problem \eqref{solutions0t} in $C(\bar{\Omega}\times [0,+\infty))\cap L^{\infty}(\bar{\Omega} \times [0,T])$.
 In addition, if (H'') holds, then the unique solution $u \in C(\bar{\Omega}\times [0,+\infty))\cap L^{\infty}(\bar{\Omega} \times [0,T])$ for all $T>0$ is uniformly bounded in $\bar{\Omega} \times [0,+\infty)$.
\end{thm}
\subsection{Large time behaviour I: convergence in the classical sense}
We address the question of the asymptotic behaviour  as $t\to \infty$ of the solution of \eqref{solutions0t}, under assumption (H''), which ensures existence and uniqueness  for the associated stationary problem. 
The main result is the following theorem.

\begin{thm}\label{primaconv}
Let $\Omega$ be an open subset of $\mathbb{R}^N$ satisfying $(O), u_0 \in C_b(\Omega)$. Assume (M), (J0), (J1). Let $H$ be  a Hamiltonian  either of Bellman type as in \eqref{bellman1} satisfying (C'), (L'), (B') or a coercive Hamiltonian  satisfying  (H0') and either (He') or (He'').
Assume (H''), (E'). Assume also that there exists a continuous function $\bar{H}\, : \, \bar{\Omega} \times \mathbb{R} \times \mathbb{R}^N \to \mathbb{R}$ satisfying
$$
H(\cdot, t,\cdot, \cdot) \to \bar{H} \quad \,\mbox{ locally uniformly in }  \bar{\Omega} \times \mathbb{R} \times \mathbb{R}^N,
$$
as $t \to \infty$.
Then, there exists a unique bounded viscosity solution $u$ for the following problem
\begin{equation}\label{stat}
\left\{
 \begin{array}{ll}
 -\mathcal{I}(u)+\bar{H}(x, u, Du)=0 \, \, &\mbox{in} \, \,  \Omega, \\
                 \frac{\partial u}{\partial n}= 0 \, \, &\mbox{on} \, \, \partial \Omega.\\
                 
                 \end{array}
\right.\,
\end{equation} 
 Moreover, the unique viscosity solution $u$ of \eqref{solutions0t} converges uniformly on compact sets in $\bar{\Omega}$ to $u_\infty$, the unique viscosity solution of the problem \eqref{stat}.
\end{thm}
For the existence and uniqueness for the problem \eqref{stat} we refer to Theorem \ref{thm} and Corollary \ref{corthm}. 
We omit the proof of Theorem \ref{primaconv} since it  is rather classical and follows the same arguments used in \cite{GT1}, where  the same kind of results have been given in the case of the Dirichlet problem for nonlocal equations (fractional Laplacian) with Hamiltonians both  coercive both of Bellman type.

\subsection{Large time behavior II: convergence to the ergodic problem}\label{secondconvergence}
In this subsection we prove large time behaviour for the problem
\begin{equation}\label{ergt}
\left\{
 \begin{array}{ll}
 \partial_t u(x)-\mathcal{I}[u(\cdot, t)](x)+H(x, Du)=0 \, \, &\mbox{in} \, \,  \Omega \times (0,+\infty),\\
                 \frac{\partial u}{\partial n}= 0 \, \, &\mbox{on} \, \, \partial \Omega \times (0,\infty),\\
                  u(x,0)=u_0(x) \, \, &\mbox{in} \, \, \bar{\Omega}.
                 \end{array}
\right.\,
\end{equation}
 where $\Omega$ is a bounded open subset of $\mathbb{R}^N$ satisfying (O), $u_0 \in C(\Omega)$ and $H$ is a Hamiltonian in superfractional coercive form, that is, $m>1$ in (H1').  
 
Existence  and uniqueness for the problem \eqref{ergt} follow from  Theorem \ref{ex}. Note that $H$ does not depend on $u$, so it does not satisfy (H').

The main result of this section is Theorem \ref{lartime}, namely the  convergence as $t \to + \infty$ of the solution of \eqref{ergt} to a solution of the ergodic problem, which we solve in Proposition \ref{solerg}. The proof of Theorem \ref{lartime} strongly relies on the H\"older regularity up to the boundary for subsolutions of \eqref{ergt} and on the control of  their oscillation. This result was proved in the stationary setting by Barles, Ley, Koike, Topp in \cite{toppreg}, Corollary $2.14$ (see also Barles and Topp  \cite{BTregcens}),  in the case of censored operators and coercive Hamiltonians with $m > 1$. We recall this result in  Proposition \ref{regtopp}. For the evolutive problem we refer to the proof of Theorem $4.5$ of \cite{toppreg}, where first Lipschitz regularity in time is proved and then Corollary $2.14$ is used to conclude the regularity in time and space.

  We  remark that, differently from \cite{BS}, where the Lipschitz regularity of the solutions is used to linearise the equations in order to apply the Strong Maximum Principle, our proof relies mainly on the  use of  a Strong Maximum Principle \`a la Coville \cite{C1}, \cite{C2} (see also Ciomaga \cite{CIO}). This means that it relies mainly  on a topological property of  the support of the
measure defining the nonlocal operator. Note that in  this final part we assume $\Omega$ bounded for technical reasons related to the proof of the Strong maximum principle, we refer to the proof of Proposition \ref{strongmaxpr}.

\subsubsection{A strong maximum principle}
We need some notation for the statement
of the Strong Maximum Principle. Let $\mu, j$ be as in the definition of the nonlocal operator  $\mathcal{I}$, that is, satisfying (M), (J0), (J1) and denote by $\mbox{supp}\mu$ the support of the measure $\mu$. 
For $x \in \mathbb{R}^n$ we define inductively
$$
X_0(x) = \{x\};\quad X_{r+1}(x)= \cup_{\xi \in X_r(x)}\{\xi + j(\xi,\mbox{supp}\{\mu_x\})\}\cap \bar{\Omega}, \quad \mbox{for } r \in \mathbb{N},
$$
and 
$$
\mathcal{X}(x)=\overline{\cup_{r\in \mathbb{N}}X_r}.
$$
The Strong Maximum Principle presented in this paragraph relies in the non locality of the operator under the "iterative covering property"
\begin{equation}\label{itcovprop}
\mathcal{X}(x)=\Omega, \quad \mbox{for all } x \in \Omega.
\end{equation}
The most basic example is the case where $j(x,z)=z$ and there exists $r>0$ such that
$
B_r\subset \mbox{supp}\{\mu\}.
$
For further details and examples we refer to \cite{toppreg}.

The following proposition states the  Strong Maximum Principle.

\begin{prop}\label{strongmaxpr}
Let $\Omega \subset \mathbb{R}^N$ be a bounded domain satisfying (O). Let $H$ be a coercive Hamiltonian  satisfying (H0') and (He'')   with $m>1$. Assume (M), (J0), (J1), (E') and \eqref{itcovprop}.  Let $u, v$ be respectively a bounded sub and  supersolution of \eqref{ergt}, such that there exists $(x_0,t_0) \in \bar{\Omega} \times (0,+\infty)$ satisfying
$$
(u-v)(x_0,t_0)=\sup_{\bar{\Omega} \times (0,+\infty)}\{u-v\}.
$$ 
Then, the function $u-v$ is constant in $\bar{\Omega} \times [0,t_0]$. 
Moreover we have
$$
(u-v)(x,t)=\sup_{x \in \bar{\Omega}}\{u(x,0)-v(x,0)\}, \mbox{ for all } (x,t) \in \bar{\Omega}\times [0,t_0].
$$
\end{prop}

The proof of  Proposition \ref{strongmaxpr} uses the following lemma, which is 
 a consequence of the comparison principle, see \cite{BS}, Theorem $4.1$.

\begin{lem}\label{k}
Let assumptions of Proposition \ref{strongmaxpr} hold.  Let $u,v$ be  bounded sub and supersolution to equation \eqref{ergt} and for $t \in [0,+\infty)$ define
$$
k(t)=\max_{\bar{\Omega}} \{u(x,t)-v(x,t)\}.
$$
Then, for all $0\leq s \leq t$, we have $k(t) \leq k(s)$.
\end{lem}

We prove the strong maximum principle. We essentially follow the argument of \cite{toppreg}, Proposition $4.1$,  with some changes due to presence of the Neumann boundary condition. We give a sketch of the proof, focusing  on the main differences.

\begin{proofsmpwb}[Proof of Proposition \ref{strongmaxpr}]\rm{
We divide the proof into several steps.
\begin{step1}-\textit{Preliminaries.}
\upshape
We want to prove that $
(u-v)(x,t)=k(0)
$ for each $(x,t) \in \bar{\Omega} \times [0,t_0]$.
Since $k(t_0)$ is a global maximum value of $k$ in $[0,+\infty)$, by Lemma \ref{k} we have $k(t)=k(0)$ for all $t \in [0,t_0]$. Then, we have just to prove that  
$$
u(x,\tau)-v(x,\tau)=k(\tau), \, \, \forall x \in \bar{\Omega}.
$$
for each $\tau \in (0,t_0)$. By upper-semicontinuity, we derive the result up to $\tau=0$ and $\tau=t_0$.
Fix $\tau \in (0,t_0)$ and define the set
\begin{equation}\label{btau}
\mathcal{B}_\tau=\{x \in \bar{\Omega} \, : \, (u-v)(x,\tau)=k(\tau)\}.
\end{equation}
We observe that by the upper-semicontinuity of $u-v$, $\mathcal{B}_\tau$ is non-empty. Then, the claim of the proposition  follows once proved that $\mathcal{B}_\tau=\bar{\Omega}$.
\end{step1}
\begin{step2}-\textit{Localization on time $\tau$.}
\upshape
For $\eta>0$, define the function
$$
(x,t)\to \Phi(x,t):=u(x,t)-v(x,t)-\eta(t-\tau)^2
$$
and note that for each $(x,t) \in \bar{\Omega} \times (0,+\infty)$  and for $\tilde{x} \in \mathcal{B}_\tau$ where $\mathcal{B}_\tau$ is defined in \eqref{btau}, we have
$
\Phi(x,t) \leq k(t)-\eta(t-\tau)^2\leq k(\tau)=(u-v)(\tilde{x},\tau)=\bar{W}(\tilde{x},\tau).
$
Then the supremum of $\Phi$ in $\bar{\Omega}\times(0,+\infty) $ is achieved and 
$$
\sup_{(x,t)\in \Omega \times (0,+\infty)}\Phi(x,t)=k(\tau).
$$
\end{step2}

\begin{step3}-\textit{Localization around a point in $\mathcal{B}_\tau$.}
\upshape
From now on we fix 
$
x_\tau \in \mathcal{B}_\tau,
$
where $\mathcal{B}_\tau$ is defined in \eqref{btau}.
We define for $\eps, \alpha>0$
$$
\psi_{\eps,\alpha}(x)=e^{-Kd(x)}\frac{|x-x_\tau|^2}{\eps^2}-\alpha d(x),
$$
where  $d$ is the signed distance from the boundary (see Remark \ref{distance}) and $K>0$  satisfies
$
K>||D^2d||_\infty+1.
$
Note that 
$
\psi_{\eps,\alpha}(x_\tau)=-\alpha d(x_\tau).
$
Moreover for each $\eps>0$ the first derivatives of $\psi_{\eps,\alpha}$ are bounded, depending on $\eps$ and $\alpha$. 

For $0<\mu<1$ we denote 
$
\omega_\mu=\mu u-v
$
and we consider
$$
(x,t) \to \Phi_\mu(x,t):=\omega_\mu(x,t)-\eta|t-\tau|^2 -(1-\mu)\psi_{\eps,\alpha}(x).
$$
By the upper-semicontinuity of $\Phi_\mu$, there exists $(x_\mu, t_\mu) \in \bar{\Omega}\times [0,t_0+1]$ such that
$$
\Phi_\mu(t_\mu,x_\mu)=\max_{\bar{\Omega}\times[0,t_0+1]} \Phi_\mu.
$$
Since $\Phi_\mu \to \Phi$ locally uniformly on $\bar{\Omega}\times [0,+\infty)$  as $\mu \to 1$, we get up to subsequences 
$$
(x_\mu,t_\mu)\to (\bar{x},\tau) \quad \mbox{ as } \mu \to 1.
$$
Not also that  for any $\alpha $ bounded 
\begin{equation}\label{barxeps}
\bar{x}=\bar{x}_{\eps} \to x_\tau \quad \mbox{ as } \eps \to 0.
\end{equation}
Indeed, by using the maximum point inequality for $\Phi_\mu$, we have
\begin{eqnarray}\label{maxw}
\Phi_\mu(x_\mu,t_\mu)\nonumber &=&(u-v)(x_\mu,t_\mu)+(\mu-1)(u+\psi_{\eps,\alpha})(x_\mu,t_\mu)-\eta|t_\mu-\tau|^2\\&\geq& k(\tau)+(\mu-1)u(x_\tau,\tau)-\alpha (\mu-1)d(x_\tau),
\end{eqnarray}
where we used the definition of $k(\tau)$.
Since $t_\mu \in [0,t_0+1]$ for all $\mu$ close to $1$, we have
$
(u-v)(x_\mu,t_\mu)\leq k(t_\mu)\leq k(\tau).
$
Coupling the previous inequality with \eqref{maxw} we get
$
\psi_{\eps,\alpha}(x_\mu)+\alpha d(x_\tau)\leq u(x_\tau,\tau)-u(x_\mu,t_\mu).
$
Then,  by the boundedness of $u$ and of $d$, for $\alpha$ small, we deduce that
$
|x_\mu-x_\tau|\leq C\eps
$
for some $C>0$  independent on $\mu$, which implies \eqref{barxeps}.
\end{step3}

\begin{step4}-\textit{Writing the viscosity inequality for $\omega_\mu$.}
\upshape
Denote $$
\phi(x,t):=(1-\mu)\psi_{\eps,\alpha}(x)+\eta(t-\tau)^2.
$$
We test $\omega_\mu$  with the function $\phi$ in $(x_\mu,t_\mu)$. We suppose $x_\mu \in \partial \Omega$, since the other case being analogous and even simpler. By the Taylor's formula for the distance function, we have
$$
n(x_\mu)(x_\mu-x_\tau)+\frac{1}{2}(x_\mu-x_\tau)^TD^2d(x_\mu)(x_\mu-x_\tau)+ o(|x_\mu-x_\tau|^2)=d(x_\tau)\geq 0
$$
and then
\begin{equation}\label{neuref}
n(x_\mu)(x_\mu-x_\tau)\geq-||D^2d||_\infty|x_\mu-x_\tau|^2/2+o(|x_\mu-x_\tau|^2).
\end{equation}
Take  $\mu, \eps$ small enough so that
\begin{equation}\label{epssmall}
1+\frac{o(|x_\mu-x_\tau|^2)}{|x_\mu-x_\tau|^2}\geq 0.
\end{equation}
By \eqref{neuref}, by the definition  of $K$ and \eqref{epssmall},  we have  
\begin{eqnarray*}
\frac{\partial \phi_{\eps,\alpha}}{\partial n}(x_\mu,t)&=&\frac{\partial \psi_{\eps,\alpha}}{\partial n}(x_\mu) \\&\geq&e^{-Kd(x_\mu)}\frac{|x_\mu-x_\tau|^2}{\eps^2}\left[K-||D^2d||_\infty+\frac{o(|x_\mu-x_\tau|^2)}{|x_\mu-x_\tau|^2}\right]+\alpha\\ &\geq&e^{-Kd(x_\mu)}\frac{|x_\mu-x_\tau|^2}{\eps^2}\left[1+\frac{o(|x_\mu-x_\tau|^2)}{|x_\mu-x_\tau|^2}\right]+\alpha>0.
\end{eqnarray*}
Then, the rest of the proof follows exactly as in \cite{toppreg}, Proposition $4.1$, by writing the viscosity inequality for $\omega_\mu$, letting first $\mu \to 1$, then $\eps \to 0$ and using the iterative covering property \eqref{itcovprop}. We give some details for completeness of exposition.
By Proposition \ref{eqmut}, we get for  $\xi >0$
\begin{multline*}
2\eta(t_\mu -\tau) -\mathcal{I}^\xi[\omega_\mu(\cdot, t_\mu)](x_\mu)-\mathcal{I}_\xi[(1-\mu)\psi_{\eps,\alpha}(\cdot)](x_\mu)\\-(1-\mu)C_{\mu,m}|D\psi_{\eps,\alpha}(x_\mu)|^m)\leq A_R (1-\mu),
\end{multline*}
where $C_{\mu,m}$ and $A_R$ are defined in Proposition \ref{eqmut}.

Since 
$$
\mathcal{I}_\xi[(1-\mu)\psi_{\eps,\alpha}(\cdot)](x_\mu)\leq (1-\mu) C'||D\psi_{\eps,\alpha}||_\infty
$$
for some $C'>0$, we get
\begin{multline}\label{par}
2\eta(t_\mu -\tau) -\mathcal{I}^\xi[\omega_\mu(\cdot, t_\mu)](x_\mu)\\-(1-\mu)\left(C'||D\psi_{\eps,\alpha}||_\infty+ (1-\mu)C_{\mu,m}|D\psi_{\eps,\alpha}(x_\mu)|^m)+A_R\right)\leq 0.
\end{multline}
We recall that $t_\mu \to \tau$ as $\mu \to 1$ and we observe that by the smoothness of $\psi_{\eps,\alpha}$ the term in parenthesis in \eqref{par} remain bounded as $\mu \to 1$. Moreover, by the Dominated Convergence Theorem, we get
$$
\mathcal{I}^\xi[\omega_\mu(\cdot, t_\mu)](x_\mu) \to \mathcal{I}^\xi[(u-v)(\cdot, \tau)](\bar{x}) \mbox{ as } \mu \to 1
$$
where  we recall that $\bar{x}$ is  the limit of $x_\mu$ as $\mu \to 1$.\\
Then
$$
\int_{\footnotesize{\begin{matrix}\bar{x}+j(\bar{x},z) \in \bar{\Omega},\\ |z|\geq \xi\end{matrix}}}(u-v)(\bar{x}+j(\bar{x},z), \tau)-(u-v)(\bar{x}, \tau)d\mu_{\bar{x}}(z)=0
$$
and letting $\eps \to 0$ and recalling that $\bar{x} \to x_\tau$ as $\eps \to 0$ and $(u-v)(x_\tau,\tau)=k(\tau)$ we finally conclude
$$
\int_{\footnotesize{\begin{matrix} x_\tau +j(x_\tau,z, \tau) \in \bar{\Omega}\\ |z|\geq \xi\end{matrix}}}(u-v)(x_\tau+j(x_\tau,z))-k(\tau)d\mu_{x_\tau}(z)=0.
$$
Since $\xi >0$ is arbitrary, we get that 
$$
(u-v)(x, \tau)-k(\tau)=0 \quad \mbox{for all } x\in X_1(x_\tau).
$$
 Therefore we can proceed in the same way as above, and conclude
by induction that 
$$
(u-v)(x, \tau)-k(\tau)=0 \quad \mbox{for all } x\in \cup_{r\in\mathbb{N}} X_r(x_\tau).
$$
 Then we conclude the proof by the upper-semicontinuity of $u-v$ and applying the iterative convering property \eqref{itcovprop}.$\Box$
\end{step4}}

\end{proofsmpwb}

\subsubsection{The ergodic problem}
Roughly speaking, solving the \textit{ergodic problem} means pass to the limit as $\d \to 0$ in the stationary  problem
\begin{equation}\label{erg}
\left\{
 \begin{array}{ll}
 \d u(x)-\mathcal{I}[u(\cdot)](x)+H(x, Du)=0 \, \, &\mbox{in} \, \,  \Omega, \\
                 \frac{\partial u}{\partial n}= 0 \, \, &\mbox{on} \, \, \partial \Omega,\\
              
                 \end{array}
\right.\,
\end{equation}
whose existence and uniqueness for $\d >0$ holds by Theorem \ref{thm}. 

We solve the ergodic problem in Proposition \ref{solerg}. Note that we need the compactness of the family of solutions $\{u_\d\}$, which relies mainly on the   regularity result for subsolutions of \eqref{erg} which we recall in the following proposition. We remark that the H\"{o}lder regularity for subsolution of the first equation of \eqref{erg} up to the boundary was proved in \cite{toppreg}, Theorem $5.5$. Note that here the Neumann boundary conditions play no role, since the regularity inside the domain follows by \cite{toppreg}, Theorem $5$, and the extension  up to the boundary  is carried out similarly to \cite{toppreg}, applying the method used by Barles in \cite{Bshpr} (see also \cite{CDLP}).

\begin{prop}\label{regtopp}
Let assumptions of Proposition \ref{strongmaxpr} hold. Then any bounded viscosity subsolution $u \, : \, \mathbb{R}^N \to \mathbb{R}$ to \eqref{erg}
is H\"{o}lder continuous in $\bar{\Omega}$ with H\"older exponent $\gamma_0=\frac{m-\sigma}{m}$ and H\"older seminorm depending on $\Omega$, the data and $\mbox{osc}_{\Omega}(u)$ and not on $\d$. 
Moreover, there exists $K>0$ such that for any bounded viscosity subsolution of \eqref{erg} we have
\begin{equation}\label{controlosc}
\mbox{osc}_{\Omega}(u)\leq K.
\end{equation}
\end{prop}

\begin{prop}\label{solerg}
Under the assumptions of Proposition \ref{strongmaxpr},  there exists a unique constant $\lambda \in \mathbb{R}$ for which the stationary ergodic problem 
\begin{equation}\label{ergeff}
\left\{
 \begin{array}{ll}
  \lambda-\mathcal{I}[u(\cdot)](x)- H(x,Du)=0 \quad &\mbox{in} \, \, \Omega,\\
\frac{\partial u}{\partial n}=0 \quad &\mbox{in} \, \, \partial \Omega
                 \end{array}
\right.\,
\end{equation}
has a solution $w \in C^{\frac{m-\sigma}{m}}(\bar{\Omega})$. Moreover $w$ is the unique solution of \eqref{ergeff} up to an additive constant.
\end{prop}

\begin{proof}
A key ingredient is  Proposition \ref{regtopp}, which gives the compactness of the family of solutions  of the approximating equation \eqref{erg}. Once we have the compactness, the proof follows standard arguments which we do not repeat. 
The uniqueness follows by  the comparison principle for \eqref{ergt} and  the application of the strong maximum principle for the problem \eqref{ergt} (Proposition \ref{strongmaxpr}).
\end{proof}

\subsubsection{Convergence as $t\to + \infty$}

\begin{thm}\label{lartime}
Let assumptions of Proposition \ref{strongmaxpr} hold. Let $u$ be the unique solution to problem \eqref{ergt}. Then, there exists a pair $(w,\lambda)$ solution to \eqref{ergeff} such that
$$
u(x,t)-\lambda t-w(x)\to 0 \mbox{ as } t \to +\infty,
$$
uniformly on $\bar{\Omega}.$
\end{thm}

We omit the details of the proof since we follow closely the arguments given in \cite{toppreg} (see also \cite{BCCI}, \cite{GT1} for the local framework and \cite{TCH} for the nonlocal one). 

We just observe that again a crucial ingredient of the proof is the H\"older regularity  of the solutions of \eqref{ergt}. The proof is carried out as in   Theorem $4.5$ of \cite{toppreg}, by first proving Lipschitz regularity in time  and then applying Proposition \ref{regtopp} to conclude the regularity in time and space.  Once established the regularity, the proof follows  by the application of the strong maximum principle proved in Proposition \ref{strongmaxpr}.



\section*{Appendix}
First we prove some lemmas used in the proof of Theorem \ref{thm}, that is, the following Lemma \ref{lemh}, Lemma \ref{estnonloclem} and Lemma \ref{primophi}. 
In  Section \ref{blowupsec} we prove the existence of the blow-up supersolution, stated in Remark \ref{proofa}, Lemma \ref{lemblowup}.

\subsection{Some lemmas used in the proof of Theorem \ref{thm}}
\begin{lem}\label{lemh}
Let $H$ be a Hamiltonian of Bellman type.  For all $\hat{s}\in \partial \Omega$, there exists $r=r(\hat{s})>0$ and $\gamma,C_2>0$ constants  such that for all $s \in \bar{B}_r(\hat{s}), \lambda \in \mathbb{R}, p \in \mathbb{R}^N$, it holds:
\begin{itemize}
\item[(i)] if $\hat{s} \in \Gamma_{\mbox{\footnotesize{out}}}$ 
\begin{equation}\label{a1} 
H(s,p-|\lambda| n(s)) \geq \gamma |\lambda|-C_2|p|-C_2; 
\end{equation}
\begin{equation}\label{a2}
H(s,p+|\lambda| n(s)) \leq -\gamma |\lambda|+C_2|p|+C_2;
\end{equation} 
\item[(ii)]  if $\hat{s} \in \Gamma$, then
\begin{equation}\label{a11} 
H(s,p-|\lambda| n(s)) \geq \gamma |\lambda|-C_2|p|-C_2; 
\end{equation}
\begin{equation}\label{a4}
H(s,p+|\lambda| n(s)) \geq \gamma|\lambda|-C_2|p|-C_2.
\end{equation}
\end{itemize}
\end{lem}

\begin{prooflhwb}
First we prove $(i)$. Since $\hat{s} \in \Gamma_{\mbox{\footnotesize{out}}}$, for $\alpha \in \mathcal{A}$, there exists $r_1,\gamma_1 >0$ small enough such that
$
b(s,\alpha)\cdot n(s)\geq \gamma_1 $ for any $s \in \bar{\Omega} \cap B_{r_1}(\hat{s})$.
Then, by the boundedness of $b$ and $l$, we get for some $C_1>0$
\begin{equation}\label{a1proof}
H(s,p-|\lambda| n(s))\geq -b(s,\alpha) \cdot p+|\lambda| b(s,\alpha) \cdot n(s)-l(s,\alpha)\geq |\lambda| \gamma_1-C_1|p|-C_1,
\end{equation}
for any $s \in \bar{\Omega} \cap B_{r_1}(\hat{s})$. 
To prove \eqref{a2}, we approximate the supremum in the Hamiltonian by a sequence $\tilde{\alpha}\in \mathcal{A}$.
In particular, we take $\tilde{\alpha}$ such that 
$$
H(s,p+|\lambda| n(s))\leq -b(s,\tilde{\alpha}) \cdot p-\lambda b(s,\tilde{\alpha}) \cdot  n(s)-l(s,\tilde{\alpha}) +1.
$$
Since  $\hat{s}\in \Gamma_{\mbox{\footnotesize{out}}}$ and using that $\mathcal{A}$ is  compact,  there exist $r_2, C>0$ small enough such that
$
b(s,\tilde{\alpha})\cdot n(s)\geq \gamma_2$ for any $s \in \bar{\Omega} \cap B_{r_2}(\hat{s}).
$
Then by the boundedness of $b$ and $l$ we get
\begin{equation}\label{a2proof}
H(s,p+|\lambda| n(s))\leq -b(s,\tilde{\alpha}) \cdot p-|\lambda| b(s,\tilde{\alpha}) \cdot  n(s)-l(s,\tilde{\alpha}) +1\leq-\gamma_2 |\lambda|+C_1|p|+C_1+1,
\end{equation}
for any $s \in \bar{\Omega} \cap B_{r_2}(\hat{s}).$  We conclude the proof of \eqref{a1} and \eqref{a2} by using \eqref{a1proof} and \eqref{a2proof} and by denoting $r=\min\{r_1,r_2\}, \gamma=\min\{\gamma_1,\gamma_2\}$ and $C_2=C_1+1$.\\
Now we prove $(ii)$. Since $\hat{s} \in \Gamma$,  there exist $r,\gamma>0$ such that 
$
b(s,\alpha_1)\cdot n(s)\geq \gamma 
$
and
$
b(s,\alpha_2)\cdot n(s)\leq -\gamma$ for any $s \in \bar{\Omega} \cap B_r(\hat{s})$.
Then we get
$$
H(s,p-|\lambda| n(s))\geq -b(s,\alpha_1) \cdot p+|\lambda| b(s,\alpha_1) \cdot n(s)-l(s,\alpha_1)\geq |\lambda |\gamma-C_1|p|-C_1,
$$
for any $ s \in \bar{\Omega} \cap B_r(\hat{s})$, proving \eqref{a11}.
Analogously, we get
$$
H(s,p+|\lambda| n(s))\geq -b(s,\alpha_2) \cdot p+|\lambda| b(s,\alpha_2) \cdot n(s)-l(s,\alpha_2)\geq |\lambda| \gamma-C_1|p|-C_1,
$$
for any $ s \in \bar{\Omega} \cap B_r(\hat{s})$ and  by denoting $C_2=C_1+1$ we conclude \eqref{a4}. \quad \quad \quad $\Box$
\end{prooflhwb}

Now we prove Lemma \ref{estnonloclem}, whose statement is given in  the proof of Theorem \ref{thm}, Step $1$.

\begin{proof}[Proof of Lemma \ref{estnonloclem}]\rm{
First we prove $(i)$. Take $-\mathcal{I}_\xi[\phi(\cdot, y)](x)$. By the definition of $\tilde{\phi}$, since $\chi_\eps, \phi$ are Lipschitz   and by (J1) we have
$$
\tilde{\phi}(x+j(x,z),y)-\tilde{\phi}(x,y)\leq C\eps^{-1}|z|+C|z|
$$
and by the first of (M) we have 
\begin{equation}\label{phiest}
\int_{\begin{matrix}x+j(x,z) \in \Omega , \\|z|<\xi\end{matrix}}\tilde{\phi}(x+j(x,z),y)-\tilde{\phi}(x,y)d\mu_x(z) \leq \eps^{-1}C\xi^{1-\sigma}.
\end{equation}
 Take now $-\mathcal{I}_\xi[u](x)$. Note that by the boundedness of $u$ we have
$$
\int_{\scriptsize{\begin{matrix}x+j(x,z) \in \Omega , \\|z|>\xi\end{matrix}}}u(x+j(x,z))-u(x)d\mu_x(z)\leq 2||u||_\infty \int_{\scriptsize{\begin{matrix}x+j(x,z) \in \Omega , \\|z|>\xi\end{matrix}}}d\mu_x(z)
$$
and then by the first of (M), we get for some $C>0$ 
\begin{equation}\label{uest}
\int_{\scriptsize{\begin{matrix},x+j(x,z) \in \Omega  \\|z|>\xi \end{matrix}}}u(x+j(x,z))-u(x)d\mu_x(z)\leq C\xi^{-\sigma}.
\end{equation}
Plugging together \eqref{phiest} and  \eqref{uest},  we conclude $(i)$ for some $C_1>0$. Similarly we prove $(ii)$.}
\end{proof}

We  state Lemma \ref{primophi} and we omit the proof since it follows by standard arguments.

\begin{lem}\label{primophi}
Let $R,\nu >0$ and denote $\psi_R(x)=\psi(R^{-1}|x|)$ where $\psi$ is a smooth function such that 
\begin{equation}\label{psir}
\psi(s)= \left\{
 \begin{array}{lll}
 0 \, \, &\mbox{for} \, \, 0 \leq s <\frac{1}{2}, \\
                 \mbox{increasing} \, \, &\mbox{for} \, \, \frac{1}{2}\leq s <1,\\
||u||_{\infty}+||v||_{\infty} +1 \, \, &\mbox{for} \, \, s \geq 1.
                 \end{array}
\right.\,
\end{equation}
Let $d$  denote the distance from the boundary of $\Omega$ in a neighbourhood $V$ of the boundary and  extend it $C^1$  and bounded in all the domain. Let $\mathcal{I}$ as in \eqref{idol} and assume (M) and (J1).
Then the function $\phi=\psi_R+\nu d$ satisfies 
\begin{equation}\label{onur}
\mathcal{I}[\phi](\cdot)\leq o_{\nu,R}(1),\quad |D\phi(\cdot)|\leq o_{\nu,R}(1).
\end{equation}
\end{lem}


\subsection{Blow-up supersolution}\label{blowupsec}
In this section we prove Lemma \ref{lemblowup}, used in the proof of Theorem 	\ref{thm} (see in particular Proposition \ref{lembell}, Remark \ref{proofa}).

We construct the function $U_r$ as showed in the following and then we prove Lemma \ref{blow}. Note that Lemma \ref{lemblowup} follows as a consequence of Lemma \ref{blow} and we prove it after the statement of Lemma \ref{blow}.

Let  $\bar{x} \in \Gamma_{\mbox{\footnotesize{in}}}$ and  $r=r(\bar{x})$ be given as in assumption (O). 
We recall that by (O), there exists a $W^{2,\infty}$-diffeomorphism 
$
\psi:B_{r}(\bar{x}) \mapsto \mathbb{R}^N, 
$
satisfying
\begin{equation}\label{limitxydelta1app}
 \psi_N(s)=d(s) \, \mbox{ for any } s \in B_{r}(\bar{x}),
\end{equation} 
where $d$ is the signed distance from the boundary of $\Omega$. 

We define  $U_r$in a suitable neighbourhood of the point $\bar{x}$, where we  rectify the boundary and carry on the computations. In particular we define $U_r$ as follows
\begin{equation}\label{UR}
U_{r}(x)=\bar{U}_{r}(d(x)) \quad \mbox{ for } x \in B_{r}(\bar{x})\cap \Omega,
\end{equation}
where 
$$
\bar{U}_{r}(s)=-\log(s)+\frac{3}{2}\log r \quad \mbox{ if } 0<s\leq r.
$$
Note that $\bar{U}_{r} \in C^\infty(0,r)$ is (non-negative) monotone and decreasing.




\begin{lem}\label{blow}
For any $\bar{x} \in \Gamma_{\mbox{\footnotesize{in}}}$, let $r=r(\bar{x})$ be defined as in assumption (O). Let $U_{r}$ be defined as in \eqref{UR}. Then  we have
for $\xi$ small enough (with respect to $r$)
\begin{equation}\label{terza0}
-\mathcal{I}_{\xi}[U_{r}](x)\geq -Ad(x)^{-\sigma}\quad \mbox{in }  B_{\frac{r}{2}}(\bar{x}) \cap \Omega.
\end{equation}
In particular there exists $\tilde{r}(r, A,\sigma)=\tilde{r}$ such that $\tilde{r}\leq r$ and
\begin{equation}\label{terza}
-b(x,\alpha)\cdot DU_{r}(x)-\mathcal{I}_{\xi}[U_r](x)\geq 0 \quad \mbox{in } B_{\frac{\tilde{r}}{2}}(\bar{x}) \cap \Omega \quad \forall \alpha \in \mathcal{A}.
\end{equation}
\end{lem}

\begin{oss}\rm{
Note that the strict positivity of the drift term on the points of $\Gamma_{\footnotesize{\mbox{in}}}$  is essential here to prove \eqref{terza}, since the drift term controls the  integral term which explodes on the boundary, as  \eqref{terza0} shows.
}
\end{oss}

As a consequence of Lemma \ref{blow}, we prove Lemma \ref{lemblowup}.
\begin{proof}[Proof of Lemma \ref{lemblowup}]
Take $\tilde{r}$ as defined in Lemma \ref{blow} and let $U_{\tilde{r}}$ be defined as in \eqref{UR} for $r=\tilde{r}$. 
Then $U_{\tilde{r}}$ is a non-negative decreasing function which trivially satisfies  (ii) of Lemma \ref{lemblowup} with $\omega_{\tilde{r}}(s)=\frac{1}{\bar{U}_{\tilde{r}}(s)}$. Moreover, (i) Lemma \ref{lemblowup} follows as a direct application of \eqref{terza} of Lemma \ref{blow}. 
\end{proof}


Now we prove Lemma \ref{blow}.

\begin{proof}[Proof of Lemma \ref{blow}]

First we prove \eqref{terza0}. Let $\xi\leq C_j^{-1}\frac{r}{2}$.
Then, by (J1), for $ |z|\leq \xi$ and $x \in B_{\frac{r}{2}}(\bar{x})$, we have that
$
x+j(x,z)\in B_{r}(\bar{x}). 
$
We  describe the domain of integration of $\mathcal{I}_{\xi}[U_r]$ through the diffeomorphism $\psi$ as 
$
x+j(x,z)\in \bar{\Omega}=\psi_N(x+j(x,z)) \geq 0.
$
By the definition of $U_r$ and \eqref{limitxydelta1app}
$$
\mathcal{I}_\xi[U_r](x)=-\int_{\footnotesize{\begin{matrix}\psi_N(x+j(x,z))\geq 0, \\|z|\leq \xi\end{matrix}}}\left[\ln(\psi_N(x+j(x,z)))-\ln(\psi_N(x))\right] d\mu_x(z).
$$
We write
$
\mathcal{I}_{\xi}[U_r](x)= I^1+I^2,
$
where
$$
I^1=-\int_{\footnotesize{\begin{matrix}\psi_N(x+j(x,z))> \psi_N(x), \\|z|\leq  \xi\end{matrix}}}\left[\ln(\psi_N(x+j(x,z)))-\ln(\psi_N(x))\right] d\mu_x(z).
$$
and
$$
I^2=-\int_{\footnotesize{\begin{matrix}\psi_N(x)\geq \psi_N(x+j(x,z))\geq 0, \\|z|\leq  \xi\end{matrix}}}\left[\ln(\psi_N(x+j(x,z)))-\ln(\psi_N(x))\right]d\mu_x(z).
$$
 and since 
 $
 I^1\leq 0,
 $
 we get
 \begin{equation}\label{blowup12}
 \mathcal{I}_{\xi}[U_r](x)\leq -\int_{\footnotesize{\begin{matrix}\psi_N(x)\geq \psi_N(x+j(x,z))\geq 0, \\|z|\leq  \xi\end{matrix}}}\left[\ln(\psi_N(x+j(x,z)))-\ln(\psi_N(x))\right]d\mu_x(z).
 \end{equation}
We proceed performing a change of variable in order to write the set of integration in terms of $\psi_N(x)$. 
In other words, we write
$$
\psi(x+j(x,z))-\psi(x)=w.
$$
Then by (J0), (J1), the first of (M) and since $\psi$ is $W^{2,\infty}$, \eqref{blowup12} becomes
$$
\mathcal{I}_{\xi}[U_r](x)\leq  \bar{C}\int_{\footnotesize{\begin{matrix}0\geq w_N\geq -\psi_N(x), \\|w|\leq C \xi\end{matrix}}}\left|\ln\left(1+\frac{w_N}{\psi_N(x)}\right)\right|\frac{dw}{|w|^{N+\sigma}},
$$
for some $\bar{C},C>0$.
By the change of variable $y=\frac{w}{\psi_N(x)}$, we get
\begin{equation}\label{cofv2}
\mathcal{I}_{\xi}[U_r](x)\leq \bar{C}\psi_N(x)^{-\sigma}\int_{0\geq y_N\geq -1}|\ln\left(1+y_N\right)|\frac{dy}{|y|^{N+\sigma}}.
\end{equation}
Note that the integral in the right hand side is finite and does not depend on $\xi$. For convenience of notation we denote
$
A:=\bar{C}\int_{\footnotesize{0\geq y_N\geq -1}}|\ln\left(1+y_N\right)|\frac{dy}{|y|^{N+\sigma}}.
$
Then \eqref{cofv2} becomes
$$
\mathcal{I_{\xi}}[U_r](x)\leq Ad(x)^{-\sigma},
$$
which is exactly \eqref{terza0}.

Now we prove \eqref{terza}.
First note that, by the definition of $U_r$, we have
\begin{equation}\label{terza00}
DU_{r}(x)=d(x)^{-1}n \quad \mbox{in } B_{C_jr}(\bar{x}) \cap \Omega.
\end{equation}
Then,  by \eqref{terza0} and \eqref{terza00}, we have for all $\alpha \in \mathcal{A}$
$$
b(x,\alpha)\cdot DU_r(x) +\mathcal{I}_{\xi}[U_r](x)\leq d(x)^{-1}(b(x,\alpha)\cdot n+ d(x)^{1-\sigma}A).
$$
Since  we are in a neighbourhood of $\Gamma_{\footnotesize{\mbox{in}}}$, $\sigma <1$ and $\mathcal{A}$ is compact, there exists $0<\tilde{r}<r$ (depending only on $A, \sigma$ and $r$), such that if $x\in B_{\frac{r}{2}}(\bar{x})\cap \Omega$
$$
b(x,\alpha)\cdot DU_r(x) +\mathcal{I}_{\xi}[U_r](x)\leq 0 \quad \forall \alpha \in \mathcal{A}.
$$
Then \eqref{terza} follows and we conclude the proof of the Lemma. 
\end{proof}

\section*{Acknowledgement}
\addcontentsline{toc}{section}{Acknowledgement}
This work has been carried out during the permanence of the author at the Laboratoire de Math\'ematiques de Tours, which the author wants to thank  for the warm welcome. In particular the author wants to express her deepest gratitude to
Prof. Guy Barles, for proposing the problem and for the important support given. Without him, this work would not have been carried out. The author kindly thanks also the anonymous referees, whose comments and suggestions were precious in the improving of the final presentation of the paper.


\begin{thebibliography}{9}
\addcontentsline{toc}{section}{Bibliography}
\bibitem{AT} O. Alvarez, A. Tourin, Viscosity solutions of nonlinear integro-differential equations. Annales de l'I.H.P. Analyse non lin\'eaire 13 (1996),  no. 3, 293-317.
\bibitem{BCD} M. Bardi, I. Capuzzo-Dolcetta, Optimal control and viscosity solutions of Hamilton-Jacobi-Bellman equations.  Birkh\"auser Boston, Inc., Boston, MA, 1997. 
\bibitem{Bshpr} G. Barles, A Short Proof of the $C^{0,\alpha}$-regularity of Viscosity Subsolutions for Superquadratic Viscous Hamilton-Jacobi Equations and Applications, Nonlinear Analysis 73 (2010), 31-47.
\bibitem{BCCI} G. Barles, E. Chasseigne, A. Ciomaga, C. Imbert, Lipschitz regularity of solutions for mixed integro-differential equations. J. of Differential Equations  252  (2012),  no. 11, 6012-6060.
\bibitem{BCCI2} G. Barles, E. Chasseigne, A. Ciomaga, C. Imbert, Large time behaviour of periodic viscosity solutions for uniformly elliptic integro-differential equations, Calc. Var. Partial Differential Equations 50, 1-2 (2014), 283-304.
\bibitem{BCGJ} G. Barles, E. Chasseigne, C. Georgeline, E. R. Jakobsen, On Neumann type problems for non-local equations set in a half space. Trans. Amer. Math. Soc.  366  (2014),  no. 9, 4873-4917.
\bibitem{BGJ} G. Barles, C. Georgeline, E. R. Jakobsen, On Neumann and oblique derivatives boundary
conditions for second-order elliptic integro-differential equations. J. of Differential Equations, 256, (2014), no. 4, 1368–1394.
\bibitem{BI} G. Barles, C. Imbert, Second-order elliptic integro-differential equations: viscosity solutions' theory revisited. Annales de l' I. H. P. Anal. Non Lin\'eaire 25, (2008), no.3, 567-585.
\bibitem{toppreg} G. Barles,  S. Koike, O. Ley, E. Topp, Regularity results and large time behaviour for integro-differential equations with coercive Hamiltonians. Calc. Var.  54  (2015),  no. 1, 539-572.
\bibitem{BL} G. Barles, P.L. Lions, Remarques sur les probl\`emes de r\'eflexion oblique. C. R. Acad. Sci.Paris S\'erie I Math. 320 (1995), no. 1, 69-74.
\bibitem{BS} G. Barles, P. Souganidis, Space-time periodic solutions and long-time behaviour of solutions of quasilinear parabolic equations. SIAM J. Math. Anal. (electronic) 32, (2001), no.6, 1311-1323.
\bibitem{BTregcens} G. Barles, E. Topp, Lipschitz regularity for censored sub diffusive integro-differential equations with superfractional gradient terms. Nonlinear Analysis: Theory, Methods and Applications, Elsevier, (2016), no. 131, 3-31.
\bibitem{GT1} G. Barles, E. Topp, Existence, uniqueness and asymptotic behaviour for nonlocal parabolic problems with dominating gradient terms. Preprint 2014, arXiv:1404.7804.
\bibitem{BBC} K. Bogdan, K. Burdzy, Z.-Q Chen, Censored Stable Processes. Prob. Theory  Related Fields 127, (2003),  no. 1, 89-152.
\bibitem{CDLP} I. Capuzzo-Dolcetta, F. Leoni, A. Porretta, H\"older Estimates in Space-Time for Viscosity Solutions of Hamilton-Jacobi Equations, CPAM 63(5): 590-529 (2010).
\bibitem{CIO} A. Ciomaga, On the strong maximum principle for second-order nonlinear parabolic
integro-differential equations. Adv. Differential Equations 17 (2012), no. 7-8, 635-671.
\bibitem{C1} J. Coville,  Maximum principles, sliding techniques and applications to nonlocal equations.
Electronic J. of Differential Equations, (2007), no. 68, 1-23.
\bibitem{C2} J. Coville,  Remarks on the strong maximum principle for nonlocal operators. Electronic
J. Differential Equations (2008) no. 66, 1-10.
\bibitem{US} M. Crandall, H. Ishii, P.-L.Lions, User's guide to viscosity solutions of second order partial differential equations. Bullettin of the American Mathematical Society 27, (1992), no.1, 1-67.
\bibitem{FS} W.H. Fleming, H. M. Soner,  Controlled Markov processes and viscosity solutions.  Springer, New York (2006).
\bibitem{FR} M. Freidlin, Functional integration and partial differential equations. Annals of Mathematics
Studies, vol. 109, Princeton University Press, Princeton, (1985).
\bibitem{GM} Q. Y. Guan, Z.M. Ma, Reflected symmetric $\alpha$-stable process and regional fractional Laplacian. Probab. Theory Related Fields 134 (2006), no.4, 649-694. 
\bibitem{ISHP} H. Ishii, Perron's method for Hamilton-Jacobi equations. Duke Math. J. 55, (1987), no.3, 369-384.
\bibitem{J}  N. Jacob, Pseudo differential operators and Markov process. Vol III. Markov process and Applications. Imperial College Press, London, (2005). xxviii-474.
\bibitem{LS} P.-L. Lions, A.-S. Sznitman, Stochastic differential equations with reflecting boundary conditions. Comm. Pure Appl. Math. 37 (1984), no. 4, 511-537.
\bibitem{SAY} A. Sayah, \'Equations d'Hamilton-Jacobi du premier ordre avec termes int\'egro-diff\'erentiels. ii. existence de solutions de viscosit\'e. Comm. Partial Differential Equations 16 (1991), no.6-7, 1075-1093.
\bibitem{TCH} T. T. Tchamba, Large time behaviour of solutions of viscous Hamilton-Jacobi equations with superquadratic Hamiltonian. Asymptot.Anal. 66 (2019), 161-186.
\bibitem{T1} E. Topp, Existence and uniqueness for integro-differential equations with dominating drift terms, Comm. Partial Differential Equations  39  (2014),  no. 8, 1523–1554.
\end{thebibliography}
\end{document}